\documentclass[11pt]{article}
\usepackage[utf8]{inputenc}
\usepackage{amsmath,amsfonts, amsthm,amssymb,bbm,mathtools,color,todonotes, stmaryrd, authblk}

\usepackage[a4paper, total={5.8in, 9.5in}]{geometry}

\PassOptionsToPackage{hyphens}{url}
\usepackage{hyperref}

\newtheorem{theorem}{Theorem}[section]
\newtheorem*{theorem*}{Theorem}
\newtheorem{cor}[theorem]{Corollary}
\newtheorem{lem}[theorem]{Lemma}
\newtheorem*{lem*}{Lemma}

\newtheorem{remark}[theorem]{Remark}
\newtheorem{definition}[theorem]{Definition}

\newtheorem{assumption}{Assumption}[section]

\numberwithin{equation}{section}
\newcommand{\RR}{\mathbb{R}} 
\newcommand{\fH}{\mathcal{H}} 
\newcommand{\EE}{\mathbb{E}} 
\newcommand{\PP}{\mathbb{P}} 
\newcommand{\fF}{\mathcal{F}} 
\newcommand{\fT}{\mathcal{T}}
\newcommand{\1}{\mathbbm{1}} 
\newcommand{\fP}[1]{\mathcal{P}_{#1}} 
\newcommand{\fL}{\mathcal{L}} 
\newcommand{\para}{\olessthan}
\newcommand{\reso}{\odot}
\newcommand{\fD}{\mathcal{D}} 
\newcommand{\fI}{\mathcal{I}}
\newcommand{\fK}{\mathcal{K}}
\newcommand{\NN}{\mathbb{N}}
\newcommand{\fS}{\mathcal{S}}

\title{The Compact Support Property of Rough Super Brownian Motion on $\RR^2$}
\author[1]{Ruhong Jin}
\author[2]{Nicolas Perkowski}
\affil[1]{University of Oxford, \href{mailto:ruhong.jin@maths.ox.ac.uk}{ruhong.jin@maths.ox.ac.uk}}
\affil[2]{Institut für Mathematik, Freie Universität Berlin,
\href{mailto:perkowski@math.fu-berlin.de}{perkowski@math.fu-berlin.de}}

\begin{document}
\maketitle

\begin{abstract}
    We discuss the compact support property of the rough super-Brownian motion constructed in \cite{perkowski2021rough} as a scaling limit of a branching random walk in static random environment. The semi-linear equation corresponding to this measure-valued process is the continuous parabolic Anderson model, a singular SPDE in need of renormalization, which prevents the use of classical PDE arguments as in \cite{englander2006compact}. But with the help of an interior estimation method developed in \cite{moinat2020space}, we are able to show that the compact support property also holds for rough super-Brownian motion.
\end{abstract}

\section{Introduction}
The aim of this paper is to give an affirmative answer to the compact support property of rough super Brownian motion introduced in the paper \cite{perkowski2021rough} as universal scaling limit of two-dimensional branching random walks in small random environments (BRWRE).

Similarly to the recent path-wise construction of solutions to singular SPDEs, such as the regularity structure introduced by Hairer \cite{hairer2014theory} and paracontrolled distribution introduced by Gubinelli, Imkeller and Perkowski \cite{gubinelli2015paracontrolled}, the rough super Brownian motion can be viewed as a 'path-wise' version of the limit process of a BRWRE with potential $V^n$ defined by
\[
    V^n(x) = \xi^n_e(x) = \xi^n(x) - c_n, \qquad \{\xi^n(x)\}_{x \in \mathbb{Z}_n^2} \text{ i.i.d}\sim n\Phi
\]
for $c_n \simeq log(n)$, a given random variable $\Phi$ with mean 0 and variance 1 and $\mathbb{Z}_n^2 = \frac{1}{n}\mathbb{Z}^2$.

In the 'path-wise' sense, with a deterministic environment $\xi$ which exhibits similar property with the samples of spatial white noise, the rough super Brownian motion $\mu$ is characterized by the following log-Laplace equation  
\begin{equation}
    \EE[e^{-\langle\mu(t),\varphi_0\rangle}] = e^{-\langle\mu(0),U_t\varphi_0\rangle}
\end{equation}
for non-negative $\varphi_0\in C_c^\infty(\RR^2)$ and $U_t\varphi_0$ defined in Section \ref{sec.rsbm}. It is also shown in \cite{perkowski2021rough} that $\mu$ is super-exponentially persistent, i.e. 
\[
    \PP\left[\lim_{t\rightarrow \infty}e^{-t\lambda}\langle\mu(t),\varphi\rangle = \infty\right] >0
\]
for all $\lambda > 0$ and nonzero positive function $\varphi\in C_c^\infty(\RR^2)$. This indicates that with positive probability, the mass of the rough super Brownian motion grows super-exponentially and eventually it will spread out to the whole space $\RR^2$. This is due to the growth of spatial white noise $\xi$ at infinity which creates more and more favorable islands the further we move from the origin. Thus it is not clear if there is a cascade of branching events where mass escapes to $\infty$ along a chain of favorable islands in finite time, or if for each finite time the rough super Brownian motion will almost surely stay in some compact set. This is the question that we address in this paper and we show that the rough super Brownian motion does possess the compact support property, just like its classical counterpart. In other words, the growth of the white noise at $\infty$ does not allow the measure-valued process to spread infinitely fast.

There are several classical results on the compact support property of the measure-valued processes, although to the best of our knowledge none for unbounded potentials. Therefore, our result is even new if we replace the white noise by a smooth stationary Gaussian process with short-range correlations . J{\'a}nos Engl{\"a}nder and Ross G. Pinsky \cite{englander1999construction,englander2006compact} consider measure-valued processes satisfying the log-Laplace equation 
\begin{equation*}
    \EE[e^{-\langle\mu(t),f\rangle}] = e^{-\langle\mu(0),u_f\rangle},
\end{equation*}
where $u_f$ is the minimal non-negative solution to the evolution equation 
\begin{equation}\label{equ.classical_parabolic}
    \left\{
        \begin{array}{ll}
             u_t = Lu + \beta u - \alpha u^2& \text{ in }\RR^d\times(0,\infty),  \\
             u(x,0) = f(x)&  \text{ in }\RR^d, \\
             u\geq 0 & \text{ in }\RR^d\times[0,\infty),
        \end{array}
    \right.
\end{equation}
where $L = \frac{1}{2}\sum_{i,j=1}^da_{i,j}\frac{\partial^2}{\partial x_i\partial x_j}  + \sum_{i=1}^db_i\frac{\partial}{\partial x_i}$ and $d$ is the dimension of the space.  Under certain conditions ($\beta$ is bounded from above with some smoothness property of $\alpha,\beta$), they relate the compact support property of $\mu$ to the uniqueness of non-negative solutions of equation $(\ref{equ.classical_parabolic})$. These results do not apply to rough super Brownian motion since the coefficient $\beta = \xi$ is only a distribution and it is unbounded from above. 

Taking a detailed look at the proofs in \cite{englander1999construction,englander2006compact} we see that the difficulty lies in  understanding the limit behaviour of solutions to the equation of zero initial value
\[
    \partial_t\varphi  = \fH\varphi - \frac{\kappa}{2}\varphi^2+\phi \qquad\text{ in }\RR_+\times\RR^2,
\]
when $\phi$ tends to infinity outside some compact domain and when this compact domain becomes larger and larger. To control these solutions, we utilize the method developed in \cite{moinat2020space,moinat2020local,chandra2023priori}, where the authors obtain interior estimates for the $\Phi_3^4$ equation. By adapting the method to our case, we are able to describe limit points as mild solutions in a suitable weighted space, for which there is uniqueness. This give an affirmative answer to the compact support property of the rough super Brownian motion.

\section{Preliminaries and Notation}

\subsection{Basic notation}
We set $\fP{n}:=[-n,n]^d$ throughout this paper and we will use $z = (t,x)$ for space-time points, while $t$ represents time and $x$ or $y$ represent spatial points only. 

For any region $D \subset \RR^d$, and $0 < T_1 < T_2$ we set $D^{T_1,T_2} := [T_1,T_2]\times D$. When $T_1 = 0$, we simply write $D^{T} = D^{0,T}$. For $r > 0$, we define $D_r:=\{x \in D: d(x,\partial D) \geq r\}$, where $d$ means spatial distance:
\begin{equation}\label{def.Spatial_Distance}
    d(x,\bar{x}) = |x-\bar{x}| := \max_{i=1,\dots, d}\left\{|x_i-\bar{x_i}|\right\}.
\end{equation}

We also write $d'$ for the space-time distance with parabolic scaling,
\begin{equation}\label{def.Parabolic_Distance}
    d'((t,x),(\bar{t},\bar{x})) := \max\left\{\sqrt{|t-\bar{t}|}, |x-\bar{x}|\right\}.
\end{equation}
Furthermore, for space-time points $z \in \RR_+ \times \RR^d$ and $R>0$, we define the ball $B'(z,R):=\{\bar{z} \in \RR_+\times\RR^d: d'(\bar{z},z) \leq R, \bar{t}\leq t\}$. We also define $B(x,R) := \{\bar{x} \in \RR^d: |x-\bar{x}|\leq R\}$ and $B(x,R,T) = B^T(x,R):= [0,T]\times B(x,R)$.

The Fourier transform of $f \in L^1(\RR^d)$ is defined as
\[
    \fF f(k) = \int_{\RR^d} e^{-2\pi i k\cdot x} f(x) dx,
\]
and it is extended to (ultra-)distributions $f$ by duality.

\subsection{Notations on the regularity}\label{sec:notation-regularity}
The symbol $\xi$ will always denote a ``typical realization'' (i.e. satisfying the regularity requirements mentioned below) of a spatial white noise or a mollified spatial white noise on $\RR^d$. We also define $\fI\xi$ by the equation 
\begin{equation}\label{def.X_xi}
     -\Delta \mathcal I \xi = \chi(\fD)\xi 
\end{equation}
for a smooth function $\chi$ which equals 1 outside of $\left(-\frac{1}{4},\frac{1}{4}\right)^d$ and equals $0$ on $\left(-\frac{1}{8},\frac{1}{8}\right)^d$. And we consider the two-variable distributions (functions in the case of a mollified white noise) formally defined as
\begin{equation}
    (\xi X)(x,\bar x) := \xi(\bar x) (X(\bar x) - X(x)),
\end{equation}
where $X(x)=x$, and
\begin{equation}
    ((\fI \xi)\xi)(x,\bar x):=(\fI\xi(\bar{x}) - \fI\xi(x))\xi(\bar{x}) - C,
\end{equation}
where $C$ is a renormalization constant. We will discuss below that for a.e. realization of the (mollified) space white noise in $d=2$ these distributions are well-defined.

To measure the regularity of distributions, we follow the approach of \cite{chandra2023priori,moinat2020local,moinat2020space}, but we are more restrictive in the choice of $\Phi$ below because this allows us to compare the regularity defined here with classical notions of regularity, which is not clear to us in the generality of \cite{chandra2023priori,moinat2020local,moinat2020space}. We fix a non-negative (and non-vanishing)  smooth and symmetric function $\tilde{\Phi}$ on $\RR^d$ with support contained in $B(0,\frac{1}{2})$. Define
\[
    \Phi := (\tilde{\Phi}*\tilde{\Phi})^2.
\]
Then $\Phi$ is a non-negative symmetric smooth function with support contained in $B(0,1)$, and $\fF\Phi$ is strictly positive on $\RR^d$. In fact, $\fF\Phi = ((\fF\tilde{\Phi})^2)*((\fF\tilde{\Phi})^2)$ and $(\fF\tilde{\Phi})^2$ is non-negative and real analytic due to the symmetry and compact support property of $\tilde{\Phi}$. Lojaciewicz's structure theorem \cite{krantz2002primer} tells us that the set of zeros of a $d$ dimensional real analytic function is of Hausdorff dimension $d-1$, hence a null set. However, if $\fF\Phi = 0$ for some point $x$, it requires the zero set of $\fF\tilde{\Phi}$ to be of infinite Lebesgue measure. This is a contradiction, which means $\fF\Phi$ is strictly positive. Multiplying with a positive constant if necessary, we may assume that the integral of $\Phi$ over $\RR^d$ equals $1$. 

Next, for $\delta>0$ we set $\Phi^\delta(x) = \delta^{-d}\Phi(\frac{x}{\delta})$ and define \[
    \Psi^{\delta,n} :=\Phi^{\delta2^{-1}}*\Phi^{\delta2^{-2}}*\cdots*\Phi^{\delta2^{-n}}\qquad \text{and} \qquad \Psi^\delta = \lim_{n\rightarrow\infty}\Psi^{\delta,n}.
\]
Then we have $\Psi^\delta = \Phi^{\frac{\delta}{2}}*\Psi^{\frac{\delta}{2}}$ and so the Fourier transform of $\Psi^\delta$ is still strictly positive. Indeed, if $\fF \Psi^\delta(z) = 0$, then from $\fF \Psi^\delta = \fF \Phi^{\frac{\delta}{2}}\fF\Psi^{\frac{\delta}{2}}$ and $\fF \Phi^{\frac{\delta}{2}} > 0$ we deduce that also $\fF \Psi^\delta(z/2) = \fF \Psi^{\delta/2}(z) =0$. Iterating this, we get $\fF \Psi^{\delta}(2^{-n}z)=0$ for all $n \in \NN$ and therefore $\fF \Psi^\delta(0) = 0$ by continuity. But this is impossible because $\fF \Psi^\delta(0) = 1$ by construction.

With the function $\Psi$, we define local norms of distributions of negative regularity. We use $(\cdot)_\delta$ to denote the convolution with $\Psi^\delta$, and $(\cdot)_{\delta,n}$ for the convolution with $\Psi^{\delta,n}$ for $n\geq 1$. For example, $f_\delta = f*\Psi^\delta$. On a set $D\subset \RR^d$ and for $\alpha < 0$, the local $\alpha-$H\"older seminorm of a distribution $f$ is defined as
	\begin{equation}\label{def.negative_holder_seminorm}
	    \|f\|_{\alpha,D} := \sup_{\delta \in (0,1]}\delta^{-\alpha} \|f_\delta\|_D,
	\end{equation}
	where $\|\cdot\|_D$ is the supremum norm over $D$. Note that $\|f\|_{\alpha,D}$ depends on $f$ on the set $B(D,1)$, but this will not influence our result since we will only estimate functions that are defined on the whole space. In particular, for the space white noise $\xi$, its multiplication with first order monomials $\xi X$ and for the related distribution $\{(\fI\xi)\xi(\cdot, \bar x)\}_{x \in \RR^d}$ we have a.s.
	\begin{equation}\label{def.xi_norm}
	    \|\xi\|_{n,-1-\epsilon} := \|\xi\|_{\fP{n},-1-\epsilon} = \sup_{x\in \fP{n}}\sup_{\delta \in (0,1]}\delta^{1+\epsilon}|\xi_\delta(x)| < \infty,
	\end{equation}
	\begin{equation}\label{def.xiX_norm}
	    \|\xi X\|_{n,-\epsilon} := \sup_{x\in \fP{n}}\sup_{\delta \in (0,1]}\delta^{\epsilon}\left|\int (\xi X)(x,\bar x) \Psi^\delta(x - \bar x) d\bar x\right| < \infty,
	\end{equation}
	\begin{equation}\label{def.enhanced_xi_norm}
	    \|(\fI\xi)\xi\|_{n,-2\epsilon}:= \sup_{x\in \fP{n}}\sup_{\delta \in (0,1]}\delta^{2\epsilon}\left|\int((\fI\xi)\xi)(x,\bar{x})\Psi^\delta(x-\bar{x})d\bar{x}\right| < \infty,
	\end{equation}
	for all $n \in \NN$ and all $\epsilon > 0$.	
	
	We will also work with functions $U(z,\bar{z})$ of two variables, and we will write $U(t,x,\bar{x}):=U((t,x),(t,\bar{x}))$. For $\alpha\in(1,2)$, the $\alpha-$H\"older semi-norm for $U(t,\cdot)$ on $D$ for points with distance less than $r$ is defined as
	\begin{equation}\label{Def.AlphaHolder.U}
	    [U(t,\cdot)]_{\alpha,D,r} := \sup_{x \in D}\inf_{\nu\in\RR^d}\sup_{\substack{\bar{x}\in D\backslash\{x\}\\ |\bar x - x|<r}} \frac{|U(t,x,\bar{x}) - \nu\cdot(\bar{x}-x)|}{|x-\bar{x}|^\alpha}.
	\end{equation}
	Note that $[U(t,\cdot)]_{\alpha,D,r}<\infty$ requires that $U(t,x,x)=0$ for all $x\in D$. And if $U(t,\cdot)$ is smooth in both its variables, then at the point $x$ the optimal $\nu$ is the spatial gradient of $U(t,x,\cdot)$ at the point $x$. 
	We also define $\bar{U}(z,\bar{z}) = U(z,\bar{z}) - \nu(x)\cdot(\bar{x}-x)$ when the spatial gradient of $U(t,x,\cdot)$ exists at point $x$ and is denoted by $\nu(x)$. As shorthand notation we define:
	\[
	    [U(t,\cdot)]_{\alpha,D} := [U(t,\cdot)]_{\alpha,D,\infty}.
	\]
	
	In addition, for $\alpha \in (0,1)$, we use $\|\cdot\|_{\alpha,D}$ (resp. $[\cdot]_{\alpha,D}$) to denote the $\alpha-$Hölder norm (resp. semi-norm) on $D$, and $\|\cdot\|_{\alpha,D,r}$ (resp. $[\cdot]_{\alpha,D,r}$) to denote the $\alpha-$Hölder norm on $D$ within spatial distance $r$. For space-time functions $f$ on $[0,T]\times D$, we define the semi-norm 
	\[
        [f]_{C_T(\alpha,D,r)}:= \sup_{0\leq t\leq T}[f(t,\cdot)]_{\alpha,D,r} = \sup_{0\leq t\leq T}\sup_{\substack{x\neq\bar{x}\in D,\\ |x-\bar{x}|<r}}\frac{|f(t,x)-f(t,\bar{x})|}{|x-\bar{x}|^\alpha}.
	\]
	To be consistent with the negative H\"older-norm defined above, we write
    \[
        [\cdot]_{\alpha,D} := [\cdot]_{\alpha,D,1}.
    \]
    This also extends to norms and to the weighted (semi-)norms that we  introduce below. On the whole space $\RR^d$, we will simply write
    \[
        \|\cdot\|_{\alpha}:=\|\cdot\|_{\alpha,\RR^d}.
    \]
	Similarly, for functions of two variables we define 
	\begin{eqnarray*}
	    [U]_{C_T(\alpha,D,r)}&:=&\sup_{0\leq t\leq T}[U(t,\cdot)]_{\alpha,D,r}.
	\end{eqnarray*}
	For simplicity, we will drop the $\alpha$ when we are valuing the supremum spatial norm.
 
\subsection{Notations on weights}
We use the same weights as in \cite{martin2019paracontrolled}.
\begin{definition}\label{def.weight}
   We write
   \[
        \omega^{pol}(x) := \log(1+|x|),\qquad\omega_\sigma^{\exp}(x) :=|x|^\sigma,
   \]
   where $x \in \RR^d,\sigma \in (0,1)$. For $\omega \in \pmb{\omega}:=\{\omega^{pol}\}\cup \{\omega_\sigma^{\exp}|\sigma\in(0,1)\}$, we denote by $\pmb{\varrho}(\omega)$ the set of measurable, strictly positive $\theta:\RR^d\rightarrow(0,\infty)$ such that for some $\lambda=\lambda(\theta) >0$ we have 
   \[
        \theta(x)\lesssim\theta(y)e^{\lambda\omega(x-y)},\qquad x,y \in \RR^d.
   \]
   We also write $\pmb{\varrho}(\pmb{\omega}):=\bigcup_{\omega\in\pmb{\omega}}\pmb{\varrho}(\omega)$. The objects $\theta \in \pmb{\varrho}(\pmb{\omega})$ are called \emph{weights}.
\end{definition}

\begin{remark}\label{rem.weight}
    We will mainly consider weights
    \[
        p(a)(x) := (1+|x|)^a,\qquad e(l)(x):=e^{l|x|^\sigma},
    \]
    for non-negative $a,l$.
\end{remark}
For any norm defined in the previous subsection, we may add a weight $\theta\in \pmb{\varrho}(\pmb{\omega})$ and consider a corresponding weighted norm. For example, we define
\[
    [f]_{C_T(\alpha,D,r,\theta)} := \sup_{t\leq T}\sup_{\substack{x,y\in D,\\|x-y|\leq r}}\frac{|f(t,x)-f(t,y)|}{\theta(x)|x-y|^\alpha},
\]
for $\alpha \in (0,1)$, and for $\alpha < 0$
\[
    \|f\|_{C_T(\alpha,D,\theta)} := \sup_{t\leq T}\sup_{0\leq\delta\leq 1}\delta^{-\alpha}\|f(t,\cdot)_\delta\theta^{-1}\|_D.
\]
This also applies to special norms that involve the noise $\xi$. For example,
	\begin{equation}
	    \|(\fI\xi)\xi\|_{-2\epsilon,\RR^d,\theta}:= \sup_{x}\sup_{\delta \in (0,1]}\delta^{2\epsilon}\left|\theta^{-1}(x)\int((\fI\xi)\xi)(x,\bar{x})\Psi^\delta(x-\bar{x})d\bar{x}\right|.
	\end{equation}

In addition, we will consider weighted Besov spaces.   Let $\rho_{-1},\rho_0 \in C_c(\RR^d)$ be two non-negative and radial functions such that the support of $\rho_{-1}$ is contained in a ball $B \subset \RR^d$, the support of $\rho_0$ is contained in an annulus $\{x \in \RR^d: 0 < a \leq |x|\leq b\}$ and such that with 
\[
    \rho_{j} = \rho_{0}(2^{-j}\cdot), \qquad j \in \NN_0,
\]
the following conditions are satisfied:
\begin{enumerate}
    \item $\sum_{i=-1}^\infty \rho_j(x) = 1$ for all $x \in \RR^d$;
    \item supp($\rho_i$) $\cap$ supp($\rho_j$) = $\emptyset$ whenever $|i-j| > 1$.
\end{enumerate}
To deal with weights in $\pmb{\varrho}(\omega^{\exp}_\sigma)$ we need to consider ultra-distributions, and in particular we need to assume (which is possible) that the partition of unity is in $\mathcal S_{\omega^{\exp}_\sigma}$, the space of smooth functions $f$ such that $f$, $\fF f$ and all their derivatives decrease faster than $e^{-\lambda \omega^{\exp}_\sigma}$ at infinity, for any $\lambda>0$. The topological dual of $\mathcal S_\omega$ is denoted by $\mathcal S_\omega'$ and it is called the space of ultra-distributions. Here we do not need to know much about ultra-distributions and we refer to \cite[Section~2.2]{martin2019paracontrolled} for a more detailed discussion.

Given $\omega \in \pmb{\omega}$ and  a partition of unity with $\rho_{-1}, \rho_0 \in \mathcal S_\omega$ we define for any ultra-distribution $f\in\fS_\omega'$ the Littlewood-Paley blocks of $f$ as 
\[
    \Delta_jf = \fF^{-1}(\rho_j\fF(f)), \qquad j \ge -1.
\]
\begin{definition}
   For $\alpha\in \RR, p,q \in[1,\infty]$ and $\theta \in \pmb{\varrho}(\pmb{\omega})$, we define the weighted Besov space $B_{p,q}^\alpha(\RR^d,\theta)$ by
   \[
        B_{p,q}^\alpha(\RR^d,\theta):= \{f \in \fS'_\omega: \|f\|_{B_{p,q}^\alpha(\RR^d,\theta)} := \|(2^{j\alpha}\|\theta^{-1}\Delta_jf\|_{L^p})\|_{\ell^q} < \infty\}.
   \]
\end{definition}
Note that we consider $\|\theta^{-1}\Delta_j f\|_{L^p}$, while in \cite{martin2019paracontrolled} it would be  $\|\theta\Delta_j f\|_{L^p}$. In the notation of \cite{martin2019paracontrolled} our weighted Besov space would be called $B_{p,q}^\alpha(\RR^d,\theta^{-1})$. 

Next, let us compare the different notions of regularity.

\begin{lem}
   Let $\alpha \in (-\infty,1)\setminus\{0\}$, $\theta\in\pmb{\varrho}(\pmb{\omega})$, and let $f\in\fS'_\omega$ be an ultra-distribution on $\RR^d$. Then $\|f\|_{\alpha,\theta}<\infty$ if and only if $f \in B^\alpha_{\infty,\infty}(\RR^d,\theta)$, and
   \[
        \|f\|_{B^\alpha_{\infty,\infty}(\RR^d,\theta)} \simeq \|f\|_{\alpha,\theta}.
   \]
\end{lem}

\begin{proof}
    We only show the inequality $\|f\|_{B^\alpha_{\infty,\infty}(\RR^d,\theta)} \lesssim \|f\|_{\alpha,\theta}$, the opposite one can be handled by similar arguments. We divide the proof into two parts. When $\alpha > 0$ and $k > -1$, we have 
    \begin{equation*}
        \Delta_k f(x) = (\fF^{-1}\rho_k) * (f - f(x))(x),
    \end{equation*}
    since $\int \fF^{-1}\rho_k(x) dx = \rho_k(0) = 0$. Using the definition of $\|f\|_{\alpha,\theta}$, we estimate
    \begin{equation*}
        \begin{aligned}
            |(&\fF^{-1}\rho_k) * (f - f(x))(x)\theta^{-1}(x)| \\
            & = \left|\int (\fF^{-1}\rho_k)(y)(f(x-y) - f(x))\theta^{-1}(x)dy\right|\\
            & \lesssim \left(\left|\int_{|y|\leq 1}|(\fF^{-1}\rho_k)(y)|\cdot|y|^\alpha dy\right|+ \left|\int_{|y|\geq 1}|(\fF^{-1}\rho_k)(y)|(1+e^{\lambda(\theta)\omega(y)})dy\right|\right)\|f\|_{\alpha,\theta}\\
            & \lesssim 2^{-k\alpha}\left(1 + \left|\int(\fF^{-1}\rho_0)(y)|y|^\alpha e^{\lambda\omega(y)}dy\right|\right)\|f\|_{\alpha,\theta} \lesssim 2^{-k\alpha}\|f\|_{\alpha,\theta}.
        \end{aligned}
    \end{equation*}
    The last inequality holds because $\rho_0 \in \mathcal{S}_\omega$. The estimate for $k=-1$ is similar and therefore the result is true for $\alpha > 0$.
    
    Next, we consider the case $\alpha <0$. The $k-$th block of $f$ is given by
    \begin{equation*}
        \begin{aligned}
             \Delta_kf &= \fF^{-1}\left(\rho_k\fF f\right)\\
             &= \fF^{-1}\left(\frac{\rho_k}{\fF\Psi^\delta}\fF f\fF\Psi^\delta\right)\\
             &=\fF^{-1}\left(\frac{\rho_k}{\fF\Psi^\delta}\right)*f_\delta,\qquad \forall 0 < \delta \leq 1,
        \end{aligned}
    \end{equation*}
    where we used that $\mathcal F\Psi^\delta$ is strictly positive. With $\delta = 2^{-k}$, we have for $k \geq 0$:
    \begin{equation*}
        \begin{aligned}
            \|\Delta_kf\cdot\theta^{-1}\|_{L^\infty} &\leq \left\|\fF^{-1}\left(\frac{\rho_k}{\fF\Psi^\delta}\right)e^{\lambda\omega}\right\|_{L^1}\|f_\delta\theta^{-1}\|_{L^\infty} \\
        &\lesssim  \left\|\fF^{-1}\left(\frac{\rho_0}{\fF\Psi^1}\right)e^{\lambda\omega}\right\|_{L^1}2^{-k\alpha},
        \end{aligned}
    \end{equation*}
    where we used that $e^{\lambda \omega(2^{-k} z)} \le e^{\lambda \omega(z)}$.
   Let us assume $\left\|\fF^{-1}\left(\frac{\rho_0}{\fF\Psi^1}\right)e^{\lambda\omega}\right\|_{L^1} < \infty$ for now, which we will show at the end of this proof. Then 
    \[
        \sup_{k\geq 0}2^{k\alpha}\|\Delta_kf\cdot\theta^{-1}\|_{L^\infty} \lesssim  \left\|\fF^{-1}\left(\frac{\rho_0}{\fF\Psi^1}\right)e^{\lambda\omega}\right\|_{L^1} < \infty.
    \]
    For $k = -1$, we choose $\delta = 1$ and obtain 
    \begin{equation*}
        \begin{aligned}
            \|\Delta_kf\cdot\theta^{-1}\|_{L^\infty} \lesssim  \left\|\fF^{-1}\left(\frac{\rho_{-1}}{\fF\Psi^1}\right)e^{\lambda\omega}\right\|_{L^1} < \infty.
        \end{aligned}
    \end{equation*}
    Thus, $f \in B^\alpha_{\infty,\infty}(\RR^d,\theta)$.
    
    Now, let us prove that $\left\|\fF^{-1}\left(\frac{\rho_0}{\fF\Psi^1}\right)e^{\lambda\omega}\right\|_{L^1} < \infty$. The case $\omega(x) = \log(1+|x|)$ follows directly from the fact that $\rho_{-1},\rho_0 \in C^\infty_c$ have compact support. For the case $\omega(x) = |x|^\sigma$, we only need to show that for any $\hat{\lambda} > 0$
    \[
        \fF^{-1}\left(\frac{\rho_0}{\fF\Psi^1}\right) \lesssim_{\hat{\lambda}}e^{-\hat{\lambda}|x|^\sigma},
    \]
    and by Lemma 3.7 in \cite{martin2019paracontrolled} this follows if for all $\delta > 0$, there is $C > 0$ such that for all $l\geq 0$ and $i=1,2$, we have 
    \[
        \left\|D^l_i\left(\frac{\rho_0}{\fF\Psi^1}\right)\right\|_{L^1} \lesssim_\delta \delta^l C^l(l!)^{\frac{1}{\sigma}}.
    \]
    Now on the support of $\rho_0$, we have, by applying Leibniz's rule to $1 = \fF\Psi^1\cdot\frac{1}{\fF\Psi^1}$,
    \[
        0 = \sum_{k=0}^l \binom{l}{k} D^k_i(\fF\Psi^1)D^{l-k}_i\left(\frac{1}{\fF\Psi^1}\right).
    \]
    By induction (make use of $\|D_i^k(\fF\Psi^1)\|_{L^\infty} \leq C_0^k$ for some constant $C_0$ since $\Psi^1$ is compactly supported), we can show that on the support of $\rho_0$, 
    \[
        \left\|D_i^k\left(\frac{1}{\fF\Psi^1}\right)\right\|_{L^\infty}\lesssim_{\delta} \delta^k C_1^k(k!)^{\frac{1}{\sigma}},
    \]
    for $\delta < 1$, and the constant in the inequality can be chosen proportional to $e^{\frac{1}{\delta}}$. Hence, again by applying Leibniz's rule to $D_i^l\left(\frac{\rho_0}{\fF\Psi^1}\right)$, it remains to bound
    \[
        \left\|\sum_{k=0}^l \binom{l}{k}(D_i^k\rho_0) \delta^{l-k}C_1^{l-k}((l-k)!)^{\frac{1}{\sigma}}B_1\right\|_{L^1}.
    \]
    Now we use that $\rho_0$ is compactly supported and that $\rho_0\in \mathcal S_\omega(\RR^d)$, which gives
    \[
        \|D_i^k\rho_0\|_{L^1} \lesssim \|D_i^k\rho_0\|_{L^\infty} \lesssim \||x|^k\hat{\rho}\|_{L^1} \lesssim \int |x|^k e^{-|x|^{\sigma}}dx \lesssim \left(\frac{k+d}{\sigma}\right)^{\lfloor\frac{k}{\sigma}\rfloor}\lesssim C_2^k (k!)^{\frac{1}{\sigma}}.
    \]
    Thus, we obtain the result.
\end{proof}


We will also measure the time regularity, so we consider also the norm
\[
    \|f\|_{\fL^{\alpha}_T(\RR^d,\theta_\cdot)} = \|f\|_{C_T(\alpha, \RR^d,\theta_\cdot)} + \|f\|_{C^{\frac{\alpha}{2}}_T(\RR^d,\theta_\cdot)},
\]
where 
\[
    \|f\|_{C^{\frac{\alpha}{2}}_T(\RR^d,\theta_\cdot)}:= \sup_{0 \leq s < t\leq T}\frac{\|f(t)- f(s)\|_{\RR^d,\theta_t}}{|t-s|^{\frac{\alpha}{2}}}.
\]
and $\theta_\cdot$ is a time-dependent weight. If we do not specify the time-dependence, then we choose $\theta_t \equiv \theta$ and just write $\fL^\alpha_T(\RR^d,\theta)$. 

\section{Rough Super Brownian Motion}\label{sec.rsbm}

Here we recall and extend the definition of the rough super-Brownian motion from \cite{perkowski2021rough}. From now on, we will work in dimension 2 and in what follows the letter $d$ will \underline{not} represent the dimension. While our arguments are dimension independent they do need regularity requirements that the white noise only satisfies in $d=2$. In higher dimensions we could consider a slightly mollified white noise and apply the same arguments, or we could treat a $3d$ white noise with the same approach but at the price of more technicalities. In $d\ge 4$ the parabolic Anderson model with white noise is scaling (super-)critical and there does not exist any solution theory for it yet, and thus in particular there exists no rough super-Brownian motion with white noise environment in $d \ge 4$.

We fix a parameter $\kappa > 0$, which will describe the strength of branching in the rough super-Brownian motion. Let $\varphi_0\in C_c^{\infty}(\mathbb{R}^2), \varphi_0\geq 0$ and $\phi \in C_c(\mathbb{R}^2),\phi\geq 0$. We use $U^{\phi}_t\varphi_0$ to denote the solution to the equation 
\begin{equation}\label{equ.PAM+square+phi}
    \left\{
    \begin{array}{cc}
         \partial_t\varphi  = \fH\varphi - \frac{\kappa}{2}\varphi^2+\phi, & \text{ in }\RR_+\times\RR^2, \\
         \varphi(0,\cdot) = \varphi_0, & \text{ on } \{0\}\times\RR^2,
    \end{array}
    \right.
\end{equation}
where $\mathcal{H} = \Delta + \xi$ is the so-called Anderson Hamiltonian; see section \ref{sec.proof_of_main_thm} for the solution theory, where it is also shown that $U^\phi_t \varphi_0 \ge 0$ for all $t \ge 0$. Just to clarify the terminology, note that for $\kappa=0$ and $\phi=0$ this equation is the parabolic Anderson model. We will discuss the solution theory of \eqref{equ.PAM+square+phi} in the next section. In particular, we use $U_t\varphi_0$ to denote the solution for $\phi = 0$, i.e.
\begin{equation}\label{equ.PAM+square}
    \left\{
    \begin{array}{cc}
         \partial_t\varphi  = \fH\varphi - \frac{\kappa}{2}\varphi^2, & \text{ in }\RR_+\times\RR^2, \\
         \varphi(0,\cdot) = \varphi_0, & \text{ on } \{0\}\times\RR^2.
    \end{array}
    \right.
\end{equation}
The definition of the rough super Brownian motion only requires making sense of equation (\ref{equ.PAM+square}), which can be solved as long as the resonant product $\fI\xi\odot \xi$ is provided. Here we write
\[
    f\para g = \sum_{i<j-1}\Delta_i f \Delta_j g,\qquad f\reso g = \sum_{|i-j|\le 1} \Delta_i \Delta_j g,
\]
for the paraproduct and the resonant product; see \cite{gubinelli2015paracontrolled} for the estimates on $\para$ and $\reso$ that we will need.

We let $\epsilon > 0$ be small enough ($\epsilon < 1/3$ suffices) and write the regularities of the ($2d$) white noise and related distributions as $-1-\epsilon$, $-2\epsilon$, etc.  We make the following assumption on the noise $\xi$.
\begin{assumption}\label{ass.rsbm}
    Let $\{\xi_\alpha\}_{\alpha\in(0,1)}$  be a family of smooth functions. We assume that there exist distributions $\xi,\fI\xi\odot\xi$ and constants $C_\alpha$ such that, with some $\epsilon' < \epsilon$,
    \[
        \lim_{\alpha\rightarrow 0} \|\xi_\alpha - \xi\|_{B^{-1-\epsilon'}_{\infty,\infty}(\RR^2,p(a))} = 0, \qquad \lim_{\alpha\rightarrow 0} \|\fI{\xi_\alpha}\odot\xi_\alpha - C_\alpha - \fI\xi\odot\xi\|_{B^{-2\epsilon'}_{\infty, \infty}(\RR^2,p(a))} = 0,
    \]
    for all $a>0$, where $-\Delta \fI{\xi_\alpha} = \chi(\fD)\xi_\alpha$.
\end{assumption}

\begin{lem}
    Suppose $\{\xi_\alpha\}_{\alpha\in(0,1)}$ is a family of smooth functions indexed by $\alpha$. Then Assumption~\ref{ass.rsbm} holds if and only if for all $a > 0$
    \[
    \lim_{\alpha\rightarrow 0} \|\xi_\alpha - \xi\|_{-1-\epsilon',\RR^2,p(a)} = 0, \qquad
    \lim_{\alpha\rightarrow 0} \|(\fI\xi_\alpha)\xi_\alpha - C_\alpha - (\fI\xi)\xi \|_{-2\epsilon',\RR^2,p(a)} = 0,
    \]
    where $(\fI\xi_\alpha)\xi_\alpha(x,\bar x) = (\mathcal I \xi_\alpha(\bar x) - \mathcal I \xi_\alpha(x))\xi_\alpha(\bar x)$.
\end{lem}

\begin{proof}
    By the equivalence of our two definitions of Besov norms we immediately get $\lim_{\alpha\rightarrow 0} \|\xi_\alpha - \xi\|_{B^{-1-\epsilon'}_{\infty,\infty}(\RR^2,p(a))} = 0$ if and only if $\lim_{\alpha\rightarrow 0} \|\xi_\alpha - \xi\|_{-1-\epsilon',\RR^2,p(a)} = 0$.

    For the second order term we first suppose that Assumption 3.1 holds. We have the decomposition 
    \begin{eqnarray*}
        &&(\fI\xi_\alpha(\bar{x})-\fI\xi_\alpha(x))\xi_\alpha(\bar{x})-C_\alpha \\
        &=& (\fI\xi_\alpha\para\xi_\alpha)(\bar{x}) + (\fI\xi_\alpha\odot\xi_\alpha)(\bar{x}) + (\xi_\alpha\para\fI\xi_\alpha)(\bar{x}) - \fI\xi_\alpha(x)\xi_\alpha(\bar{x}) - C_\alpha 
    \end{eqnarray*}
    From the theory of paracontrolled products, see e.g. \cite[Lemma~4.2]{martin2019paracontrolled}, we know that $\xi_\alpha\para\fI\xi_\alpha$ converges in $B^{-2\epsilon'}_{\infty, \infty}(\RR^d, \theta)$. Together with the equivalence of the two Besov norms, this shows that
    \[
        \sup_{x\in\RR^2}\sup_{\delta\in(0,1]}\left| p_a(x)^{-1}\int(\fI\xi_\alpha\odot\xi_\alpha - C_\alpha + \xi_\alpha\para\fI\xi_\alpha - \fI\xi\odot\xi - \xi\para\fI\xi)(\bar{x})\Psi^\delta(x-\bar{x})d\bar{x}\right|\delta^{2\epsilon'}
    \] 
    converges to $0$ as $\alpha \to 0$. It remains to check that
    \begin{equation}\label{equ.para_est}
        \sup_{x\in\RR^2}\sup_{\delta\in(0,1]}\left|\int(\fI\xi_\alpha\para\xi_\alpha(\bar{x})-\fI\xi_\alpha(x)\xi_\alpha(\bar{x}) - (\fI\xi\para\xi(\bar{x})-\fI\xi(x)\xi(\bar{x})))\Psi^\delta(x-\bar{x})d\bar{x}\right|\delta^{2\epsilon'}
    \end{equation}
    converges to $0$.
    We now consider modelled distributions $f^{\pi}(x)=\fI\xi(x)\Xi$ resp. $f^{\pi}(x)=\fI\xi_\alpha(x)\Xi$ in the regularity structure of the (regularized) PAM such that $\Pi_x\Xi = \xi$ resp. $\Pi_x \Xi = \xi_\alpha$, see \cite{hairer2014theory} or Section~6 of~\cite{gubinelli2015paracontrolled} for more details. Since the regularity of the modelled distribution is smaller than 0, Theorem 6.10 in \cite{gubinelli2015paracontrolled} shows that the reconstruction operator can be taken as the paraproduct $\para$. Thus, the convergence in \eqref{equ.para_est} holds. The converse direction follows from similar  arguments, so we have the equivalence.
\end{proof}

    Assumption~\ref{ass.rsbm} is an alternative version of Assumption 2.3 (``Deterministic Environment'') in \cite{perkowski2021rough}, which guarantees the well-posedness of the PAM and hence the existence of the rough super-Brownian motion. The assumption is satisfied for almost all sample paths of the space white noise on $\RR^2$, see \cite{perkowski2021rough}.

Let $\mathcal M(\RR^2)$ be the space of finite positive Borel measures on $\RR^2$, equipped with the topology of weak convergence.

\begin{definition}\label{def.rsbm}
	Let $\xi$ satisfy Assumption \ref{ass.rsbm}. Let $\kappa > 0$ and let $\mu$ be a process with values in the space $C([0,\infty), \mathcal{M}(\mathbb{R}^2))$, such that $\mu(0)$ is compact supported. Write $\mathcal{F} = \{\mathcal{F}_t\}_{t\in[0,\infty)]}$ for the completed and right-continuous filtration generated by $\mu$. We call $\mu$ a rough super-Brownian motion(rSBM) with parameter $\kappa$ if it satisfies one of the two following equivalent properties: 
    \begin{enumerate}
        \item 	For any $t\geq 0$ and $\varphi_0 \in C_c^{\infty}(\mathbb{R}^2), \varphi_0\geq 0$ and for $U_{.}\varphi_0$ the solution to equation (\ref{equ.PAM+square}) with initial condition $\varphi_0$, the process 
	\[
		N_t^{\varphi_0}(s) = e^{-\langle\mu(s),U_{t-s}\varphi_0\rangle}, \qquad s \in [0,t]
	\]
	is a bounded continuous $\mathcal{F}-martingale$.
	\item For any $t\geq 0$ and $\varphi_0\in C_c^{\infty}(\mathbb{R}^2)$ and $f \in C([0,t]; C^{\zeta}(\mathbb{R}^2, e(l)))$ for some $\zeta > 0$ and $l < -t$, and for $\varphi_t$ solving 
	\begin{equation}\label{inverse.pam.f}
	    \partial_s\varphi_t + \mathcal{H}\varphi_t = f, \qquad s \in [0,t], \qquad \varphi_t(t) = \varphi_0.
	\end{equation}
	it holds that 
	\[
	    s\rightarrow M_t^{\varphi_0,f} := \langle\mu(s),\varphi_t(s)\rangle - \langle\mu(0),\varphi_t(0)\rangle - \int_0^s dr\langle\mu(r),f(r)\rangle,
	\]
	defined for $s \in [0,t]$, is a continuous square-integrable $\mathcal{F}-$martingale with quadratic variation 
	\[
	    \langle M_t^{\varphi_0,f}\rangle_s = \kappa\int_0^s dr\langle \mu(r),(\varphi_t)^2(r)\rangle.
	\]
    \end{enumerate}
\end{definition}

In order to discuss the compact support property of the rough super-Brownian motion, we will need to generalize the first definition of the process.
\begin{lem}
    The definition of the rough super-Brownian motion is equivalent to the following property: For any $t\geq 0$ and $\varphi_0 \in C_c^{\infty}(\mathbb{R}^2),\phi\in C_c^{\infty}(\RR^2), \varphi_0,\phi\geq 0$
    and for $U_{.}^{\phi}\varphi_0$ the solution to equation (\ref{equ.PAM+square+phi}) with initial condition $\varphi_0$, the process 
	\[
		N_t^{\varphi_0,\phi}(s) = e^{-\langle\mu(s),U^{\phi}_{t-s}\varphi_0\rangle -\int_0^sdr\langle\mu(r),\phi\rangle}, \qquad s \in [0,t]
	\]
	is a bounded continuous $\mathcal{F}-martingale$. In particular, the rough super-Brownian motion satisfies
    \begin{equation}\label{eq.rsbm_exp_equ}
        \EE\left[e^{-\langle\mu(t),\varphi_0\rangle-\int_0^t\langle\mu(r),\phi\rangle dr}\right] = e^{-\langle\mu(0),U_t^\phi\varphi_0\rangle}.
    \end{equation}
\end{lem}
\begin{proof}
    Clearly the condition is sufficient, because it is stronger than our first characterization of the rough super-Brownian motion. To see that it is also necessary, consider $\varphi_t(s) = U_{t-s}^\phi\varphi_0$ with time independent function $\phi \in C^\infty_c(\RR^d)$. Then $\varphi_t$ satisfies equation (\ref{inverse.pam.f}) with $f = \frac{\kappa}{2}(\varphi_t)^2 - \phi$. Hence, It\^o's formula applied to $F = e^{-x}\in C^2(\RR^d)$ yields
    \begin{eqnarray*}
        N_t^{\varphi_0,\phi}(s) &=& F\left(\langle\mu(s),\varphi_t(s)\rangle + \int_0^sdr\langle\mu(r),\phi\rangle\right) \\
        &=& N_t^{\varphi_0,\phi}(0) - \int_0^s N_t^{\varphi_0,\phi}(r)\langle\mu(r),\frac{\kappa}{2}(\varphi_t)^2(r)\rangle dr\\
        &+& \frac{1}{2}\int_0^s N_t^{\varphi_0,\phi}(r)d\langle M_t^{\varphi_0,f}\rangle(r) -\int_0^sN_t^{\varphi_0,\phi}(r)dM_t^{\varphi_0,f}(r).
    \end{eqnarray*}
    Now since
    \[
        \langle M_t^{\varphi_0,f}\rangle_s = \kappa\int_0^s dr\langle\mu(r),(\varphi_t)^2(r)\rangle,
    \]
    the drift terms vanish and $N^{\varphi_0, \phi}$ is a local martingale. As $N^{\varphi_0,\phi}$ is bounded it is a true martingale.
\end{proof}

The compact support property of a measure-valued process is formulated in the following definition.
\begin{definition}\label{def.compact_support_property}
    Suppose $\mu$ is a stochastic process with values in $C(\RR_+,\mathcal{M}(\RR^2))$. We say that $\mu$ possesses the compact support property if for all $t\ge 0$
    \[
        \PP\left[\left(\bigcup_{0\leq s \leq t}\operatorname{supp} \mu(s)\right) \text{ is bounded}\right] = 1.
    \]
\end{definition}
Now we are ready to state our main result.
\begin{theorem}\label{thm.rsbm_compact_support_property}
    The rough super-Brownian motion on $\RR^2$ possesses the compact support property.
\end{theorem}
The proof is given in Section \ref{sec.proof_of_main_thm}. Let us formulate several necessary steps now, following Engl\"ander and Pinsky \cite{englander1999construction, englander2006compact}. Recall that $\fP{n} = [-n,n]^2$ and consider regions $\fP{n}^m := \left(-n-\frac{1}{m},n+\frac{1}{m}\right)^2$. Let $\phi^n_m \in C^\infty_c$ be such that 
        \[
            \phi_n^m(x)  =
                \begin{cases}
                     0, & x \in \fP{n}\cup\fP{n+2}^c, \\
                     m, & x \in (\fP{n}^m)^c\cap\fP{n+1},\\
                     \phi_n^m(x) \in [0,m], & \text{elsewhere.}
                \end{cases}
        \]
        Then consider $\varphi_n^m(t):= U_t^{\phi_n^m}0$, where $0$ is the function with value $0$ everywhere. Define $A_t^n:=\{\mu(s)((\fP{n})^c) = 0, s\leq t\}$. Since the rough super-Brownian motion has continuous trajectories in $\mathcal M(\RR^2)$ it follows from \eqref{eq.rsbm_exp_equ} that
        \begin{eqnarray*}
            \PP\left[A_t^n\right] &=& \lim_{m\rightarrow\infty}\EE\left[\exp\left(-\int_0^t \langle\mu(s),\phi_n^m\rangle ds\right)\right]\\
            &=& \lim_{m\rightarrow\infty}\exp\left(-\langle\mu(0),\varphi_n^m(t)\rangle\right).
        \end{eqnarray*}
        Hence 
        \begin{align*}
            \MoveEqLeft
           \PP\left[\left(\bigcup_{0\leq s \leq t} \operatorname{supp} \mu(s)\right) \text{ is bounded}\right] \\
           &= \lim_{n\rightarrow\infty} \PP[A_t^n] =\lim_{n\rightarrow\infty}\lim_{m\rightarrow\infty}\exp\left(-\langle\mu(0),\varphi_n^m(t)\rangle\right).
        \end{align*}
    Thus, our goal is to show that $\lim_n \lim_m \varphi^m_n = 0$. In \cite{englander1999construction} this is based on explicit supersolutions, which works because they consider bounded $\xi$. In our case we do not know any explicit supersolutions, and instead we will derive nonlinear interior estimates for~\eqref{equ.PAM+square+phi} by adapting the methods of~\cite{moinat2020local} for $\Phi^4_3$  to our setting. We leave the discussion of the well-posedness of~\eqref{equ.PAM+square+phi} to Section \ref{sec.proof_of_main_thm} and discuss the interior estimates first.

\section{Interior estimates}
Here we consider positive solutions to the equation
\begin{equation}\label{equ.renormalized.PAM+square}
        \left\{
        \begin{array}{cc}
            \partial_t u  = (\Delta + \xi - C)u - \frac{\kappa}{2}u^2, &  \text{ in }\mathbb{R}_+\times \fP{n},\\
             u(0,\cdot) = 0,& \text{ in } \{0\} \times \fP{n},
        \end{array}
    \right.
\end{equation}
where $\xi$ is a typical realization of a mollified white noise on $\RR^2$, $C$ is a renormalization constant, and we do not specify any boundary conditions on $\partial \fP{n}$. We will use the nonlinearity $-\frac{\kappa}{2} u^2$ to derive uniform bounds for $u$ on the interior of $\fP{n}$ that only depend on distributional norms of $\xi$ and $(\mathcal I \xi) \xi -C$ but not on the mollification. Most ideas in this section come from the works of Chandra, Moinat and Weber, \cite{moinat2020local,chandra2023priori,moinat2020space}, with some changes to enable estimations near time $0$. We define the two-variable functions
\begin{equation}\label{def.U}
    U(z,\bar{z}):= u(\bar{z}) - u(z) - u(z)(\fI\xi(\bar{x})-\fI\xi(x)),
\end{equation}
and 
\begin{equation}\label{def.bar_U}
    \bar{U}(z,\bar{z}) := U(z,\bar{z}) - \nu(z)(\bar{x} - x),
\end{equation}
where $\nu(z)$ is the spatial derivative of $U(z,\cdot)$ at $z$. Note that we are working with regularized noise, so all the functions we encounter are smooth and the derivative $\nu(z)$ exists. Of course, our goal is to derive estimates that are uniform in the mollification parameter.

The main estimate in this section is:
	\begin{theorem}\label{thm.local_est.renormalized.PAM+square}
	    Let $n\in\mathbb{N}$, $0 < l < n$. If $u$ solves equation (\ref{equ.renormalized.PAM+square}) in $[0,T]\times \fP{n}$, then we have:
        \begin{equation}\label{equ.est.PAM+square}
            \|u\|_{C_T\fP{n-l}} \lesssim \max\left\{\frac{1}{l^2}, \|\tau\|_{n,|\tau|}^{\frac{2}{n_\tau(1-\epsilon)}}: \tau \in \mathcal T\right\}
        \end{equation}
        for $\mathcal T = \{\xi, (\fI\xi)\xi, \xi X\}$ and $\fP{n-l} = [-n+l,n-l]^2$. Here $|\tau|$ is the regularity of $\tau$ and $n_\tau$ is the number of nodes in the tree. The implicit constant in ``$\lesssim$'' only depends on $\kappa,\Psi$ and $\epsilon$.
	\end{theorem}
        The proof is inspired by the paper \cite{moinat2020space}. We divide the proof into several lemmas. As a direct consequence, we get a bound for weighted norms of $u$.
	\begin{cor}\label{cor.local_est_weighted.renormalized.PAM+square}
	    Let $r \in \mathbb N$ and let $\theta \in \pmb{\varrho}(\omega)$ be a weight such that $\theta$ is a radial function that is increasing in $|x|$ and such that $\theta(0)\ge 1$. Let $u$ solve the equation (\ref{equ.renormalized.PAM+square}) in $[0,T]\times\fP{r+1}$. Then
        \begin{equation}
            \|u\|_{C_T(\fP{r},\theta^{\frac{2}{1-\epsilon}})} \lesssim 1 + \sum_{\tau \in \mathcal T} \|\tau\|_{r+1,|\tau|,\theta}^{\frac{2}{n_\tau(1-\epsilon)}}.
        \end{equation}
	\end{cor}
 \begin{proof}
     We take $n=r+1$ and $l=1$ in the inequality \eqref{equ.est.PAM+square}, so that
     \[
        \|u\|_{C_T\fP{r}} \lesssim \max\left\{1, \|\tau\|_{r+1,|\tau|}^{\frac{2}{n_\tau(1-\epsilon)}}: \tau \in \mathcal T\right\}.
     \]
     Then we take the $\theta^{\frac{2}{1-\epsilon}}$ weighted supremum of $u$ between the radius $r$ and $r+1$, which is controlled by 
     \[
        e^{\frac{2}{1-\epsilon}\lambda\omega(1)}\max\left\{\theta(r+1)^{-1}, \theta(r+1)^{-\frac{2}{n_\tau(1-\epsilon)}}\|\tau\|_{r+1,|\tau|}^{\frac{2}{n_\tau(1-\epsilon)}}: \tau \in \mathcal T\right\}.
     \]
     Since $\theta$ is increasing, we obtain the result.
 \end{proof}
	
 In the proof of Theorem \ref{thm.local_est.renormalized.PAM+square} we will not only derive bounds on norms of $u$, but also bounds on norms of related functions such as $U,\nu$ (see for example \eqref{equ.U_Holder_est}, which is true without Assumption \eqref{equ.ass}, and choose carefully the quantity $d_0$ to get an estimate on $[U]_{C_T(2-2\epsilon, D_d,d)}$). Moreover, we can also control the regularity of $u$, $\nu$ and $U$, see Corollary~\ref{cor.est_all} below.
 
	The proof of Theorem \ref{thm.local_est.renormalized.PAM+square} will be based on a series of lemmas whose proofs are in Appendix~\ref{app}. To start with, we give an interior supremum norm estimate for a simplified equation. 

	\begin{lem}\label{lem.general_equ}
		Let $T>0$ and let $u\in C^\infty$ solve
  	     \begin{equation}\label{equ.general}
		  \left\{
			\begin{array}{ll}
				(\partial_t - \Delta)u = - u^2 + g & \text{ in } [0,T]\times \fP{n}, \\
				u = 0 & \text{ in } \{0\}\times \fP{n}, \\
				u\geq 0,
			\end{array}
		  \right.
	   \end{equation}
	where $g$ is a smooth and bounded function. Then the following point-wise bound on $u$ holds for all $z = (t,x) \in [0,T)\times \fP{n}$:
		\begin{equation}
			u(t,x) \leq 28\cdot \max\left\{\frac{1}{\min\{(n-x_i)^2,(n+x_i)^2, i = 1,2\}},\sqrt{\|g\|_{C_T \fP{n}}}\right\}.
		\end{equation}
	\end{lem}
    The proof is on p.~\pageref{pr.lem.general_equ} in the appendix. Together with a mollification, this lemma already gives a first estimate for the equation \eqref{equ.renormalized.PAM+square}.
    \begin{lem}\label{lem.u_supremum_est}
        Let $u$ be solution of the equation \eqref{equ.renormalized.PAM+square} with $\kappa = 2$. Suppose $R>R'$ and $r+R < n$, we have for any $0 < \delta < 1$, 
        \begin{equation}\label{equ.est.u_supremum_first}
	        \begin{aligned}
	            \|u\|_{C_T\fP{n-r-R}}& \lesssim \max\left\{\frac{1}{(R-R')^2},\delta^{1-\epsilon}[u]_{C_T(1-\epsilon,\fP{n-r-R'+\delta},\delta)},\right.\\
	            &\hspace{50pt} \|u\|_{C_T(\fP{n-r-R'+\delta})}^{\frac{1}{2}}\delta^{\frac{1-\epsilon}{2}}[u]^{\frac{1}{2}}_{C_T(1-\epsilon,\fP{n-r-R'+\delta},\delta)},\\
	            &\hspace{50pt} \left.\|(u(\xi - C))_\delta\|^{\frac{1}{2}}_{C_T(\fP{n-r-R'})}\right\}.
	       \end{aligned}
	    \end{equation}
    \end{lem}
    \begin{proof}
        We first convolute the equation (\ref{equ.renormalized.PAM+square}) with $\Psi^\delta$ to obtain    \begin{equation}\label{equ.renormalized.convoluted.PAM+square}
	        (\partial_t - \Delta)u_\delta = (u\xi)_\delta - Cu_\delta - (u^2)_\delta.
	\end{equation}
        Rewrite the equation (\ref{equ.renormalized.convoluted.PAM+square}) to 
	    \[
	        (\partial_t - \Delta)u_\delta = -u_\delta^2 + u_\delta^2 - (u^2)_\delta + (u\xi)_\delta - Cu_\delta.
	    \]
        
	    Then the Lemma \ref{lem.general_equ} shows
	    \begin{equation*}
	        \begin{aligned}        \|u_\delta\|_{C_T(\fP{n-r-R})} \lesssim \max&\left\{(R-R')^{-2}, \|u_\delta^2 - (u^2)_\delta\|^{\frac{1}{2}}_{C_T(\fP{n-r-R'})},\right.\\
	        &\quad\left.\|(u(\xi - C))_\delta\|^{\frac{1}{2}}_{C_T(\fP{n-r-R'})}\right\}.
	        \end{aligned}
	    \end{equation*} 
	   Since for any region $D\subset \RR^2$, 
	    \[
	        \|u_\delta - u\|_{C_T(D_\delta)} \lesssim_{\Psi} \delta^{1-\epsilon} [u]_{C_T(1-\epsilon,D,\delta)},
	    \]
	    and 
	    \begin{eqnarray*}
	        &&(u_\delta^2 - (u^2)_\delta)(t,x) \\
	        &=& \int\Psi^\delta(x-y)(u^2_\delta(t,x) - u^2(t,y))dy\\
	        &=& 2\int\Psi^\delta(x-y)(u_\delta(t,x) - u(t,y)) \int(\lambda u_\delta(t,x) + (1-\lambda)u(t,y))d\lambda dy\\
	        &\lesssim&\|u(t,\cdot)\|_{B(x,\delta)}\int\Psi^\delta(x-y)(u_{\delta}(t,x) - u(t,x)+u(t,x)-u(t,y))dy\\
	        &\lesssim& \|u(t,\cdot)\|_{B(x,\delta)} \int\Psi^\delta(x-y)(\delta^{1-\epsilon} + d(x,y)^{1-\epsilon})[u(t,\cdot)]_{1-\epsilon,B(x,\delta),\delta}dy\\
	        &\lesssim& \|u\|_{C_T(B(x,\delta))}\delta^{1-\epsilon}[u]_{C_T(1-\epsilon,B(x,\delta),\delta)},
	    \end{eqnarray*}
	    we then have the result.
    \end{proof}

    In order to estimate the norms, we need to adjust two lemmas from \cite{moinat2020space} to our setting. The first one is a variant of Hairer's reconstruction theorem~\cite{hairer2014theory}. Recall the notation $B^T(x,R) = [0,T] \times B(x,R)$, where $B(x,R)$ is the closed ball with radius $R$ and center $x$, and that $\Psi^\delta$ is the function from Section~\ref{sec:notation-regularity} with which we measure regularity.
    
	\begin{lem}[\cite{moinat2020space}, Theorem~2.8]\label{lem.reconstruction}
	    Let $\gamma >0$ and let $A$ be a finite subset of $(-\infty,\gamma]$. Let $T > 0,\delta \in [0,1)$ and $x \in \RR^2$. Let $F:B^T(x,\delta)^2\rightarrow\RR$ be continuous and such that for all $\beta\in A$ there exist constants $C_\beta > 0$ and $\gamma_\beta\geq \gamma$ such that for all $\bar{\delta}\in(0,\delta)$, for all $x_1\in B(x,\delta -\bar{\delta})$ and $x_2 \in B(x_1,\bar{\delta})$
	    \begin{equation}
	        \left|\int \Psi^{\bar{\delta}}(x_2 - y)(F(t,x_1,y) - F(t,x_2,y))dy\right| \leq \sum_{\beta \in A}C_\beta d(x_1,x_2)^{\gamma_\beta - \beta}\bar{\delta}^\beta,
	    \end{equation}
	    uniformly in $t \in [0,T]$. Then $f(t,x):= F(t,x,x)$ satisfies for all $t \in [0,T]$
	    \begin{equation}
	        \left|\int\Psi^{\delta}(x - y)(F(t,x,y) - f(t,y))dy\right|\lesssim \sum_{\beta\in A}C_\beta \delta^{\gamma_\beta},
	    \end{equation}
	    where the implicit constant only depends on $A$ and $\gamma$.
	\end{lem}
 
	To be precise, the formulation of this lemma is slightly different than in \cite{moinat2020space}  because the time variable appears differently. But the proof from \cite{moinat2020space} works verbatim in our case, we just have to freeze $t \in [0,T]$. As a consequence, we get the following estimate, where we recall that
    \begin{align*}
        [u(t,\cdot)]_{\alpha,D,r,\theta} = & \sup_{\substack{x\neq \bar x \in D,\\ |x-\bar x|<r}} \frac{|u(t,x) - u(t,\bar x)|}{\theta(x)|x-\bar x|^\alpha},\\
        [U(t,\cdot)]_{\beta,D,r,\theta} := & \sup_{x \in D}\inf_{\nu\in\RR^d}\sup_{\substack{\bar{x}\in D\backslash\{x\}\\ |\bar x - x|<r}} \frac{|U(t,x,\bar{x}) - \nu\cdot(\bar{x}-x)|}{\theta(x)|x-\bar{x}|^\beta},
    \end{align*}
    for $\alpha \in (0,1)$ and for $\beta \in (1,2)$.
 
	\begin{lem}\label{lem.est.irregular_term}
    Let $u$ be the solution to the equation \eqref{equ.renormalized.PAM+square}. The following bounds holds for any weight $\theta \in\pmb{\varrho}(\pmb{\omega})$ and for $z = (t,x)$ with $x \in \mathcal P_n$:
	\begin{equation}\label{equ.est.irregular_term}
	        \begin{aligned}
                    \MoveEqLeft |((u-u(z))\xi)_\delta(z) - C u_\delta(z)|\theta^{-2}(x) \\
	                & \lesssim   e^{2\lambda(\theta)\omega(\delta)}\delta^{1-3\epsilon} \Big([u(t,\cdot)]_{1-\epsilon,B(x,\delta),\delta,\theta} \|\fI\xi\xi\|_{n,-2\epsilon,\theta} \\
                    & \hspace{80pt} + [U(t,\cdot)]_{2-2\epsilon,B(x,\delta),\delta,\theta}
	                \|\xi\|_{n,-1-\epsilon,\theta}\\
                    & \hspace{80pt} +[\nu(t,\cdot)]_{1-2\epsilon,B(x,\delta),\delta,\theta}\|\xi X\|_{n,-\epsilon,\theta}\Big)\\
                    & \quad + |u(z)|\theta^{-1}(x)\|\fI\xi\xi\|_{n,-2\epsilon,\theta}\delta^{-2\epsilon} + |\nu(z)|\theta^{-1}(x)\|\xi X\|_{n,-\epsilon,\theta} \delta^{-\epsilon},
	        \end{aligned}
	    \end{equation}
     where, by a small abuse of notation, we write $\omega(\delta) = \omega(x)$ for some (and thus all) $x$ with $|x|=\delta$.
	\end{lem}
 \begin{proof}[Proof of Lemma \ref{lem.est.irregular_term}]\label{pr.lem.est.irregular_term}
	    Recall that $U(z,\cdot) = u - u(z) - u(z)(\mathcal I\xi - \mathcal I \xi(x))$ and $\bar U(z,\cdot) = U(z,\cdot) - \nu(z)(\cdot - x)$ and therefore
	    \begin{align} \label{eq:est.irregular_term.pr1} \nonumber
	        \MoveEqLeft((u-u(z))\xi)_\delta(z) - Cu_\delta(z) \\ \nonumber
	        &= (\bar{U}(z,\cdot)\xi)_\delta(z) + C(u(z) - u_\delta(z)) + u(z)(((\fI\xi - \fI\xi(z))\xi)_\delta - C) \\
	        &\quad + \nu(z)((\cdot-x)\xi)_\delta. 
	    \end{align}
	    For $z=(t,x)$ with $x \in \fP{n}$ we bound the last two terms on the right hand-side by
	    \begin{equation}\label{eq:est.irregular_term.pr2}
	        \theta^{-2}(x)|u(z)(((\fI\xi- \fI\xi(z))\xi)_\delta-C)(z)| \leq \theta^{-1}(x)|u(z)|\|\fI\xi\xi\|_{n,-2\epsilon,\theta}\delta^{-2\epsilon}
	    \end{equation}
	    and
	    \begin{equation}\label{eq:est.irregular_term.pr3}
	        \theta^{-2}(x)|\nu(z)((\cdot - x)\xi)_\delta(z)| \leq \theta^{-1}(x)|\nu(z)|\|\xi X\|_{n,-\epsilon,\theta} \delta^{-\epsilon}.
	    \end{equation}
        To control the remaining terms in~\eqref{eq:est.irregular_term.pr1} let
        \[
            F(t,x,\bar x) = \bar U(z,\bar z)\xi(\bar x) + C(u(z) - u(\bar z)),
        \]
        where $\bar z = (t,\bar x)$. Then
        \begin{align*}
	        F(t,x_1,y) - F(t,x_2,y) &= (u(z_2) - u(z_1))((\fI\xi(y) - \fI\xi(x_2))\xi(y) - C) \\
	        &\quad +\bar{U}(t,x_1,x_2)\xi(y) + (\nu(z_2)-\nu(z_1))(y-x_2)\xi(y).
	    \end{align*}
	    Hence, for $\bar \delta \in (0,\delta)$ and for $x_1 \in B(x,\delta - \bar{\delta}),x_2 \in B(x_1,\bar{\delta})$, we have
	    \begin{align*}
	        \MoveEqLeft
            \int\Psi^{\bar{\delta}}(x_2-y)(F(t,x_1,y) - F(t,x_2,y))dy \\
	        &\leq \theta(x_2)^2 \Big([u(t,\cdot)]_{1-\epsilon,B(x,\delta),\delta,\theta}d(x_1,x_2)^{1-\epsilon}\|\fI\xi\xi\|_{n,-2\epsilon,\theta}\bar \delta^{-2\epsilon}\\
	        &\hspace{50pt} +[U(t,\cdot)]_{2-2\epsilon,B(x,\delta),\delta,\theta}d(x_1,x_2)^{2-2\epsilon}\|\xi\|_{n,-1-\epsilon,\theta}\bar \delta^{-1-\epsilon}\\
	        &\hspace{50pt}  +[\nu(t,\cdot)]_{1-2\epsilon,B(x,\delta),\delta,\theta}d(x_1,x_2)^{1-2\epsilon}\|\xi X\|_{n,-\epsilon,\theta}\bar \delta^{-\epsilon}\Big)
	    \end{align*}
	    Thus, using that $f(t,x) = F(t,x,x) = 0$, we obtain from Lemma \ref{lem.reconstruction} the bound
	    \begin{align*}
	        \MoveEqLeft
            \theta^{-2}(x)|(\bar{U}(z,\cdot) \xi)_\delta (z) + C(u(z)-u_\delta(z))| \\
            & = \theta^{-2}(x)\Big|\int \Psi^\delta(x - y) F(t,x,y) dy\Big|\\
	        &\leq e^{2\lambda(\theta)\omega(\delta)}\delta^{1-3\epsilon} \Big([u(t,\cdot)]_{1-\epsilon,B(x,\delta),\delta,\theta}\|\fI\xi\xi\|_{n,-2\epsilon,\theta} \\ 
            &\hspace{80pt} + [U(t,\cdot)]_{2-2\epsilon,B(x,\delta),\delta,\theta}\|\xi\|_{n,-1-\epsilon,\theta}\\
            &\hspace{80pt}+[\nu(t,\cdot)]_{1-2\epsilon,B(x,\delta),\delta,\theta}\|\xi X\|_{n,-\epsilon,\theta}\Big).
	    \end{align*}
     Plugging this back into~\eqref{eq:est.irregular_term.pr1}, together with~\eqref{eq:est.irregular_term.pr2} and~\eqref{eq:est.irregular_term.pr3}, we obtain the claimed bound.
	\end{proof}
    As a direct consequence we have a uniform estimate for the mollified, renormalized product:
    \begin{cor}\label{cor.renormalized.est}
        Let $u$ be the solution to the equation \eqref{equ.renormalized.PAM+square} and $\theta \in\pmb{\varrho}(\pmb{\omega})$. We have for $D_\delta \subset \mathcal P_n$
        \begin{equation}\label{main.step2}
	        \begin{aligned}
	            &\|(u\xi - Cu)_\delta\|_{C_T(D_\delta,\theta^2)} \\
	            \lesssim_\theta&\quad\delta^{1-3\epsilon}([u]_{C_T(1-\epsilon,D,\delta,\theta)}\|\fI\xi\xi\|_{n,-2\epsilon,\theta} \\
	            &+[U]_{C_T(2-2\epsilon,D,\delta,\theta)}\|\xi\|_{n,-1-\epsilon,\theta}+[\nu]_{C_T(1-2\epsilon,D,\delta,\theta)}\|\xi X\|_{n,-\epsilon,\theta})\\
	            &+\|u\|_{C_T(D_\delta,\theta)}(\|\fI\xi\xi\|_{n,-2\epsilon,\theta}\delta^{-2\epsilon}+\|\xi\|_{n,-1-\epsilon,\theta}\delta^{-1-\epsilon})\\
	            &+\|\nu\|_{C_T(D_\delta,\theta)}\|\xi X\|_{n,-\epsilon,\theta}\delta^{-\epsilon}.
	            \end{aligned}
	\end{equation}
    \end{cor}
    \begin{proof}
        We have 
             \begin{equation*}
        \begin{aligned}
            &\quad\|(u\xi - Cu)_\delta\|_{C_T(D_\delta,\theta^2)} \\
            &\lesssim \|((u-u(z))\xi)_\delta(z) - Cu_\delta(z)\|_{C_T(D_\delta,\theta^2)} + \|u\xi_\delta\|_{C_T(D_\delta,\theta^2)},
        \end{aligned}
     \end{equation*}
    so that the claim follows from lemma \ref{lem.est.irregular_term} and because $\|\xi_\delta\|_{C_T(D_\delta, \theta)} \lesssim \|\xi\|_{n,-1-\epsilon,\theta}\delta^{-1-\epsilon}$.
    \end{proof}
 
    We also need to adapt a Schauder estimate for functions of two variables from \cite{moinat2020space} (a variant of Hairer's Schauder estimates for modelled distributions \cite{hairer2014theory}):
    
	\begin{lem}\label{lem.Schauder_est.two_variable}
	   Let $T > 0$, $\kappa \in (1, 2)$ and let $A\subset (-\infty,\kappa]$ be finite. Let $U$ be a bounded function of two variables defined on a domain $D^T\times D^T$ such that $U(z,z) = 0$ for all $z \in D^T$ and $U(0,\cdot,\cdot) = 0$. Let $d_0 > 0$ and assume that for any $0 < d \leq d_0$ and $L\leq \frac{d}{4}$ there exists a constant $M_{D_d,L}^{(1)}$ such that for all base points $x \in D_d$ and length scales $\delta\leq L$, it holds uniformly in $t$ that 
	   \begin{equation}\label{equ.Schauder_est.two_variable.ass1}
	       \delta^2\|(\partial_t - \Delta)U_\delta(t,x,\cdot)\|_{B(x,L)}\leq M_{D_d,L}^{(1)}\sum_{\beta\in A}\delta^\beta L^{\kappa - \beta}.
	   \end{equation}
	   Assume furthermore, that for $L_1,L_2\leq \frac{d}{4}$ there exists a constant $M_{D_d,L_1,L_2}^{(2)}$ such that for any $x \in D_d$, for any $x_1 \in B(x,L_1)$ for any $x_2 \in B(x_1,L_2)$ the following 'three-point continuity' holds:
	   \begin{equation}\label{equ.Schauder_est.two_variable.ass2}
	       |U(t,x,x_2) - U(t,x,x_1) - U(t,x_1,x_2)| \leq M_{D_d,L_1,L_2}^{(2)}\sum_{\beta\in A}d(x_1,x)^\beta d(x_2,x_1)^{\kappa-\beta}.
	   \end{equation}
	   Additionally, define 
	   \[
	        M^{(1)}:=\sup_{d\leq d_0}d^\kappa M_{D_d,\frac{d}{4}}^{(1)},\qquad M^{(2)}:=\sup_{d\leq d_0}d^\kappa M_{D_d,\frac{d}{4},\frac{d}{4}}.
	   \]
	   Then 
	   \begin{equation}\label{equ.Schauder_est.two_variable.result}
	       \sup_{d\leq d_0} d^\kappa [U]_{C_T(\kappa,D_d,d)} \lesssim M^{(1)} + M^{(2)} + \sup_{d\leq d_0}\|U\|_{D_d^T,d}.
	   \end{equation}
	   Here, $'\lesssim'$ denotes a bound that holds up to a multiplicative constant that only depends on $\kappa$ and $A$ and $T$.
	\end{lem}
 
	The proof is on p.~\pageref{pr.lem.Schauder_est.two_variable}. As a direct consequence we get an estimate for $\nu$, with the same proof as in \cite{moinat2020space}:
	
	\begin{cor}\label{cor.est.spatial_gradient}
	    Fix $0\leq d\leq d_0$ such that $D_d\neq \emptyset$. Assume that $D_d$ satisfies an interior cone condition with parameter $d\geq r_d > 0$ and $\lambda \in (0,1)$, i.e. for all $r \in [0,r_d]$, for all $x \in D_d$, for any vector $\nu \in \RR^2$, there exists $y \in D_d$ such that $d(x,y) = r$ and
	    \[
	        |\nu\cdot(y-x)| \geq \lambda|\nu|d(x,y).
	    \]
	    Then for the optimal function $\nu$ in inequality (\ref{equ.Schauder_est.two_variable.result}), for all $r \in [0,r_d]$,
	    \[
	        \lambda\|\nu(t,\cdot)\|_{D_d}\leq [U(t,\cdot)]_{\kappa,D_d,d}r^{\kappa - 1}+ \|U(t,\cdot)\|_{D_d,r}r^{-1}.
	    \]
	    If inequality (\ref{equ.Schauder_est.two_variable.ass2}) holds for all $t \in [0,T], x,x_1,x_2\in D_d$, we have for $r \leq r_d$,
	    \[
	        [\nu(t,\cdot)]_{\kappa-1} \lesssim [U(t,\cdot)]_{\kappa,D_d,d} + M_{D_d,\frac{d}{4},\frac{d}{4}}^{(2)} + r^{-\kappa}\|U(t,\cdot)\|_{D_d,r}.
	    \]
	    The implicit constant in this inequality only depends on $\lambda,\kappa$ and $A$.
	\end{cor}
    Recall the definition \eqref{def.U} and the definition \eqref{def.bar_U}
    \[
        U(z,\bar{z}):= u(\bar{z}) - u(z) - u(z)(\fI\xi(\bar{x})-\fI\xi(x)),
    \]
    we apply lemma \ref{lem.Schauder_est.two_variable} and the corollary \ref{cor.est.spatial_gradient} to the function $U$ and obtain the estimation 
    \begin{lem}\label{lem.U_holer_est}
        Let $u$ be solution of the equation \eqref{equ.renormalized.PAM+square} with $\kappa = 2$. For any region $D\subset \fP{n}$, and $d_0 > 0$ such that $D$ contains a ball of radius $d_0$, we have \begin{equation}\label{equ.U_Holder_est}
	        \sup_{d\leq d_0}d^{2-2\epsilon}[U]_{C_T(2-2\epsilon,D_d,d)} \lesssim \sup_{d\leq d_0}\{A\|u\|_{C_TD} + Bd^{2-2\epsilon}[U]_{C_T(2-2\epsilon,D_d,d)}\},
	    \end{equation}
	    where 
	    \begin{eqnarray*}
                \begin{aligned}
	        A = &d^2\|u\|_{C_TD} + d^2J(\xi,\chi) + d^{1-\epsilon}(\|\xi X\|_{n,-\epsilon} + [\fI\xi]_{n,1-\epsilon} + \|\xi\|_{n,-1-\epsilon}) \\
	        +& d^{2-2\epsilon}([\fI\xi]_{n,1-\epsilon}\|\xi\|_{n,-1-\epsilon} + [\fI\xi]_{n,1-\epsilon}^2 + \|\xi X\|_{n,-\epsilon}[\fI\xi]_{n,1-\epsilon} + \|\fI\xi\xi\|_{n,-2\epsilon}) \\
	        +& d^{3-3\epsilon}([\fI\xi]_{n,1-\epsilon}\|\fI\xi\xi\|_{n,-2\epsilon} + [\fI\xi]_{n,1-\epsilon}^2\|\xi X\|_{n,-\epsilon}) + 1, 
                \end{aligned}
	    \end{eqnarray*}
	    and 
	    \begin{eqnarray*}
	        B =&& d^{1-\epsilon}(\|\xi\|_{n,-1-\epsilon}+[\fI\xi]_{n,1-\epsilon} + \|\xi X\|_{n,-\epsilon}) \\
	        &+& d^{2-2\epsilon}(\|\fI\xi\xi\|_{n,-2\epsilon} + \|\xi X\|_{n,\epsilon}[\fI\xi]_{n,1-\epsilon}).
	    \end{eqnarray*}
    \end{lem}
    \begin{proof}
        For any $L>0, 0<\delta < 1$, let $x \in D_{L+\delta}, \bar{x} \in B(x,L) \subset D_\delta$. We have 
	        \begin{equation*}
	            \begin{aligned}
	                &(\partial_t - \Delta)U(t,x,\cdot)_\delta(\bar{x}) \\
	                =&\quad \int\Psi^\delta(\bar{x} - \tilde{x})(\partial_t-\Delta_{\tilde{x}})U(t,x,\tilde{x})d\tilde{x}\\
	                =&\quad(u(\xi-C))_\delta(\bar{z}) - u(z)\xi_\delta(\bar{z})-(u^2)_\delta(\bar{z})+u(z)((1 - \chi(\fD))\xi)_\delta(\bar{z})\\
	                =&\quad((u-u(\bar{z}))\xi - Cu)_\delta(\bar{z}) - (u^2)_\delta(\bar{z}) + (u(\bar{z}) - u(z))\xi_\delta(\bar{z})\\
	                &+u(z)((1 - \chi(\fD))\xi)_\delta(\bar{z}).
	            \end{aligned}
	        \end{equation*}
	        Since $(1-\chi)$ is compactly supported, $(1-\chi(\fD))\xi$ is smooth and the supremum norm is bounded by a constant $J(\xi,\chi)$. For other terms, we have estimations
	        \[                                  |(u^2)_\delta(\bar{z})| \leq \|u(t,\cdot)\|^2_{B(\bar{x},\delta)} \leq \|u\|^2_{C_TD},
	        \]
	        and 
	        \begin{eqnarray*}
	            |(u(\bar{z}) - u(z))\xi_\delta(\bar{z})| &\leq& d(x,\bar{x})^{1-\epsilon}[u(t,\cdot)]_{1-\epsilon,B(x,L),L}\|\xi\|_{n,-1-1\epsilon}\delta^{-1-\epsilon}\\
	            &\leq& L^{1-\epsilon}[u(t,\cdot)]_{1-\epsilon,B(x,L),L}\|\xi\|_{n,-1-\epsilon}\delta^{-1-\epsilon}.
	        \end{eqnarray*}
	        For $x,L,\delta,d$, s.t. $x \in D_d, \delta \leq L\leq \frac{d}{4},d\leq d_0$, we have 
	        \begin{equation*}
	            \begin{aligned}
	                &\|(\partial_t -\Delta)U(t,x,\cdot)_\delta\|_{B(x,L)} \\
	                \lesssim& \|u\|_{C_TD}^2 + L^{1-\epsilon}[u]_{C_T(1-\epsilon,D_\delta,L)}\|\xi\|_{n,-1-\epsilon}\delta^{-1-\epsilon} \\
	                &+ \delta^{1-3\epsilon}([u]_{C_T(1-\epsilon,B(x,L+\delta),\delta)}\|\fI\xi\xi\|_{n,-2\epsilon} \\
	            &+[U]_{C_T(2-2\epsilon,B(x,L+\delta),\delta)}\|\xi\|_{n,-1-\epsilon}+[\nu]_{C_T(1-2\epsilon,B(x,L+\delta),\delta)}\|\xi X\|_{n,-\epsilon})\\
	            &+\|u\|_{C_TB(x,L)}\|\fI\xi\xi\|_{n,-2\epsilon}\delta^{-2\epsilon} +\|\nu\|_{C_T(B(x,L))}\|\xi X\|_{n,-\epsilon}\delta^{-\epsilon} + J(\xi,\chi)\|u\|_{C_TD}.
	            \end{aligned}
	        \end{equation*}
	    For $L_1,L_2 \leq \frac{d}{4},x \in D_d,x_1 \in B(x,L_1), x_2 \in B(x_1,L_2)$, we have
	    \begin{equation*}
	        \begin{aligned}
	            &|U(t,x,x_2) - U(t,x,x_1) - U(t,x_1,x_2)|\\ 
	            =& |(u(t,x_1) - u(t,x))(\fI\xi(x_2) - \fI\xi(x_1))| \\
	            \leq& [u(t,\cdot)]_{1-\epsilon,D_{\frac{d}{2}},\frac{d}{2}}[\fI\xi]_{n,1-\epsilon}d(x_1,x)^{1-\epsilon}d(x_1,x_2)^{1-\epsilon}.
	        \end{aligned}
	    \end{equation*}
	    With $\kappa = 2-2\epsilon$ in the Lemma \ref{lem.Schauder_est.two_variable}, we obtain 
	    \begin{equation}\label{equ.est.Holder_U.1}
	        \begin{aligned}
	            &\sup_{d\leq d_0} d^{2-2\epsilon}[U]_{\mathcal{M}_T(2-2\epsilon,D_d,d)} \\
	            \lesssim&\sup_{d\leq d_0}\{d^2\|u\|_{C_TD}^2 + d^2J(\xi,\chi)\|u\|_{C^TD}+ [u]_{C_T(1-\epsilon,D_{\frac{d}{2}},\frac{d}{2})}\|\xi\|_{n,-1-\epsilon}d^{2-2\epsilon} \\
	        &+d^{3-3\epsilon}([u]_{C_T(1-\epsilon,D_{\frac{d}{2}},\frac{d}{2})}\|\fI\xi\xi\|_{n,-2\epsilon} + [U]_{C_T(2-2\epsilon,D_{\frac{d}{2}},\frac{d}{2})}\|\xi\|_{n,-1-\epsilon} \\
	        &+[\nu]_{C_T(1-2\epsilon,D_{\frac{d}{2}},\frac{d}{2})}\|\xi X\|_{n,-\epsilon}) + d^{2-\epsilon}\|\nu\|_{C_T(D_{\frac{d}{2}},\frac{d}{2})}\|\xi X\|_{n,-\epsilon} \\
	        &+ d^{2-2\epsilon} \|u\|_{C_T(D_{\frac{d}{2}})}\|\fI\xi\xi\|_{n,-2\epsilon}\} + \sup_{d\leq d_0}\|U\|_{C_TD_d,d}.
	        \end{aligned}
	    \end{equation}
	    By corollary \ref{cor.est.spatial_gradient}, we have inequalities for $d\leq d_0$
	    \begin{equation}\label{equ.est.nu_supremum}
	        \|\nu\|_{C_TD_d} \lesssim [U]_{C_T(2-2\epsilon,D_d,d)}d^{1-2\epsilon} + \|U\|_{C_T(D_d,d)}d^{-1},
	    \end{equation}
	    and 
	    \begin{equation}\label{equ.est.nu_holder}
	        \begin{aligned}
	        [\nu]_{C_T(1-2\epsilon,D_d,d)}\lesssim& [U]_{C_T(2-2\epsilon,D_d,d)} + d^{-2+2\epsilon}\|U\|_{C_T(D_d,d)} \\
	        &+ [u]_{C_T(1-\epsilon,D_{\frac{d}{4}},\frac{d}{4})}[\fI\xi]_{n,1-\epsilon}.
	        \end{aligned}
	    \end{equation}
	    We have furthermore,
	    \begin{equation}\label{equ.est.u_holder}
	        [u]_{C_T(1-\epsilon,D_d,d)} \lesssim [U]_{C_T(2-2\epsilon,D_d,d)}d^{1-\epsilon} + \|u\|_{C_TD_d}[\fI\xi]_{n,1-\epsilon} + \|\nu\|_{C_TD_d}d^\epsilon,
	    \end{equation}
	    and 
	    \begin{equation}\label{equ.est.U_supremum}
	        \|U\|_{C_T(D_d,d)} \lesssim 2\|u\|_{C_TD_d} + \|u\|_{C_TD_d}[\fI\xi]_{n,1-\epsilon}d^{1-\epsilon}.
	    \end{equation}
	    Now substitute the above four inequalities (\ref{equ.est.nu_supremum} - \ref{equ.est.U_supremum}) for equation $(\ref{equ.est.Holder_U.1})$, and we obtain the result.
    \end{proof}

    With the help of these lemmas, we are now ready to give the proof of Theorem \ref{thm.local_est.renormalized.PAM+square}. We start with a lemma
    \begin{lem}\label{lem.contra.itera.est}
        There exist constants $c_1 \in (0,1)$ and $C_0 >0$ that only depend on $\epsilon, J(\xi,\chi), \Psi, T$ and such that the following statement is true:
        Suppose for some $r>0$ we have
        \[
            c_0\|u\|_{C_T(\fP{n-r})}\geq \max\left\{\|\tau\|_{n,|\tau|}^{\frac{2}{n_\tau(1-\epsilon)}}: \tau = \xi,\xi X,\fI\xi\xi\right\}
        \]
        for some $c_0 \le c_1 $. Then we have for any $R < n-r$,
        \begin{equation}
             \|u\|_{C_T\fP{n-r-R}} \leq \max\left\{\frac{2C_0^2}{R^2},\frac{1}{2}\|u\|_{C_T(\fP{n-r})}\right\}.
        \end{equation}
    \end{lem}
    \begin{proof}
        Let $D = \fP{n-r}$. We start by estimating $A,B$ in lemma \ref{lem.U_holer_est}. By choosing $d_0 = \|u\|_{C_TD}^{-\frac{1}{2}} \leq 1$ we have
        \[
	        A \lesssim 1 + J(\xi,\chi) + c_0^{\frac{1-\epsilon}{2}} + c_0^{1-\epsilon} + c_0^{\frac{3(1-\epsilon)}{2}} \lesssim 1 + J(\xi,\chi) + c_0^{\frac{1-\epsilon}{2}},
	    \]
	    and 
	    \[
	        B \lesssim c_0^{\frac{1-\epsilon}{2}}.
	    \]
        There then exists a constant $c_1 < 1$ such that whenever $c_0 \leq c_1$ we have the inequality
	    \begin{equation}\label{equ.est.U_Holder_to_u}
	        \sup_{d\leq d_0}d^{2-2\epsilon}[U]_{C_T(2-2\epsilon,D_d,d)} \lesssim (1 +J(\xi,\chi)+ c_0^{\frac{1-\epsilon}{2}})\|u\|_{C_TD}
	    \end{equation}
	    where $'\lesssim'$ does not depends on the value of $c_0$ as long as $c_0\leq c_1$. Then by inequalities $(\ref{equ.est.nu_supremum} - \ref{equ.est.U_supremum})$, we have the estimation
	    \begin{equation}\label{equ.est.all_to_u}
	        \sup_{d\leq d_0}\{\|U\|_{D_d,d},d^{1-\epsilon}[u]_{C_T(1-\epsilon,D_d,d)},d\|\nu\|_{C_TD_d},d^{2-2\epsilon}[\nu]_{C_T(1-2\epsilon,D_d,d)}\}\lesssim \|u\|_{C_TD}.
	    \end{equation}
        Now we combine lemma \ref{lem.u_supremum_est}, lemma \ref{lem.est.irregular_term} with $\theta = 1$ and equations \eqref{equ.est.all_to_u},\eqref{equ.est.U_Holder_to_u}  with $D_d = \fP{n-r-R'+\delta},R'=d_0,\delta=\frac{d_0}{k},d = d_0\frac{k-1}{k}$ for some $k>2$, to obtain 
	    \begin{equation}
	    \begin{aligned}
	        \|u\|_{C_T\fP{n-r-R}} \lesssim& \max\left\{\frac{1}{(R-R')^2},\|u\|_{C_T\fP{n-r}}((k-1)^{\epsilon-1} + (k-1)^{\frac{\epsilon-1}{2}}),\right.\\
	        &\|u\|_{C_T\fP{n-r}}\cdot
	        \left(c_0^{\frac{1-\epsilon}{2}}(k-1)^{\frac{\epsilon-1}{2}}k^\epsilon + c_0^{\frac{1-\epsilon}{4}}k^{\frac{1+\epsilon}{4}}(k-1)^{\epsilon-1} +c_0^{\frac{1-\epsilon}{2}}k^\epsilon\right.\\ 
	        &+ \left.\left.c_0^{\frac{1-\epsilon}{4}}k^{\frac{1+\epsilon}{2}} + c_0^{\frac{1-\epsilon}{4}}(k-1)^{-\frac{1}{2}}k^{\frac{1+\epsilon}{2}}\right)\right\}.
	    \end{aligned}
	    \end{equation}
	    Now first choose $k$ large enough and then $c_0$ small enough such that we finally have for a constant $C_0 > 0$
	    \[
	        \|u\|_{C_T\fP{n-r-R}} \leq \max\left\{\frac{C_0^2}{(R-R')^2},\frac{1}{2}\|u\|_{C_T\fP{n-r}}\right\}\leq \max\left\{\frac{2C_0^2}{R^2},\frac{1}{2}\|u\|_{C_T\fP{n-r}}\right\}.
	    \]
	    Here $c_0,k,C_0$ only depends on $\epsilon,J(\xi,\chi),\Psi$, and $T$. In particular, $C_0$ does not depend on $n,r,R$. 
    \end{proof}
    Now we can finish the proof of theorem \ref{thm.local_est.renormalized.PAM+square}.
    \begin{proof}[Proof of Theorem \ref{thm.local_est.renormalized.PAM+square}]
	    Without loss of generality we take $\kappa=2$ in equation (\ref{equ.renormalized.PAM+square}).  Suppose for some $D \subset \fP{n}$, we have 
	    \begin{equation}\label{equ.ass}
	        c_0\|u\|_{C_T(D)} \geq \max\left\{\|\tau\|_{n,|\tau|}^{\frac{2}{n_\tau(1-\epsilon)}}: \tau = \xi,\xi X,\fI\xi\xi\right\}
	    \end{equation}
	    for some $c_0 \le c_1$ with $c_1$ from lemma \ref{lem.contra.itera.est}. Then this lemma gives us
	    \[
	        \|u\|_{C_T\fP{n-R}} \leq \max\left\{\frac{2C_0^2}{R^2},\frac{1}{2}\|u\|_{C_T\fP{n}}\right\},
	    \]
	    if \eqref{equ.ass} holds on $D = \fP{n}$. In particular, if $C_0 > 1 + \frac{\sqrt{2}}{2}$ and $R = \frac{2C_0}{\|u\|_{C_T\fP{n}}}$, then we have 
	    \[
	        \|u\|_{C_T\fP{n-R}} \leq \frac{1}{2}\|u\|_{C_T\fP{n}}
	    \]
	    which is true for all $n \in \RR^+$ and $D = \fP{n}$ satisfies \eqref{equ.ass}. Now we begin with $R_0 = 0$ and choose $R_i$ recursively via
	    \[
	        R_{i+1}- R_i = 2c_0\|u\|_{C_T\fP{n-R_i}},
	    \]
	    except that we stop if $R_{i} > n$ (set $R_i = n$ in this case) or if $D = \fP{n}$ does not satisfy \eqref{equ.ass}. Suppose we obtain the sequence $0=R_0 < R_1 < \cdots < R_N$. We have for all $k \leq i \leq N-1$,
	    \[
	        \|u\|_{C_T\fP{n-R_i}} \leq \left(\frac{1}{2}\right)^{i-k}\|u\|_{C_T\fP{n-R_k}}.
	    \]
	    Thus for $i\leq N-1$, we have
	    \begin{eqnarray*}
	        R_i = \sum_{k=0}^{i-1}R_{i+1} - R_i &=& 2C_0\sum_{k=0}^{i-1}\|u\|_{C_T\fP{n-R_k}}^{-\frac{1}{2}} \\
	        &\leq& 2C_0\sum_{k=0}^{i-1}\|u\|_{C_T\fP{n-R_i}}^{-\frac{1}{2}}2^{-\frac{i-k}{2}}\\
	        &\lesssim& 2C_0\|u\|_{C_T\fP{n-R_i}}^{-\frac{1}{2}}.
	    \end{eqnarray*}
	    Thus we know $\|u\|_{C_T\fP{n-R_i}} \lesssim \frac{1}{R_i^2}$. Now for $R \in (R_i,R_{i+1})$, we have 
	    \[
	        R \leq R_{i+1} = R_{i+1} - R_i + R_i \leq 2C_0\|u\|_{C_T\fP{n-R_i}}^{-\frac{1}{2}} + R_i \lesssim \|u\|_{C_T\fP{n-R_i}}^{-\frac{1}{2}} \leq \|u\|_{C_T\fP{n-R}}^{-\frac{1}{2}}.
	    \]
	    This ends the proof.
    \end{proof}

By a slight extension of the arguments we can also control the regularity of $u$:
    
\begin{cor}\label{cor.est_all}
	    We have for $\tilde{\theta} = \theta^{\frac{4}{1-\epsilon}}$
	    \begin{equation}
	        \|u\|_{C_T(1-\epsilon,\fP{r},\tilde{\theta})} \lesssim 1 + \sum_{\tau \in \mathcal T} \|\tau\|_{r+1,|\tau|,\theta}^{\frac{4}{n_\tau(1-\epsilon)}}.
	    \end{equation}
	    The same bound holds for $\|\nu\|_{C_T(1-2\epsilon,\fP{r},\tilde{\theta})}, \|U\|_{C_T(\fP{r},\tilde{\theta})}$ and $[U]_{C_T(2-2\epsilon,\fP{r},\tilde{\theta})}$.
\end{cor}
 \begin{proof}
     Let $C = \max\left\{\|\tau\|_{n,|\tau|}^{\frac{2}{n_\tau(1-\epsilon)}}: \tau \in \mathcal T \right\}$, choose $d_0 = c_0C^{-\frac{1}{2}}$ and substitute it into the inequality \eqref{equ.U_Holder_est} with $D = \fP{n-\frac{1}{2}}$. This leads to
     \[
        c_0^{2-2\epsilon}C^{-1+\epsilon}[U]_{C_T(2-2\epsilon,D_{d_0},d_0)} \lesssim_{c_0} \|u\|_{C_T(D)},
     \]
     for $c_0 < c_1$, where $c_1$ is defined in lemma \ref{lem.contra.itera.est} and does not depend on $n$. To estimate the Hölder norm of $u$, we combine equations \eqref{equ.est.nu_supremum},\eqref{equ.est.U_supremum} and equation \eqref{equ.est.u_holder} to get 
     \begin{equation*}
         [u]_{C_T(1-\epsilon,D_d,d)} \lesssim [U]_{C_T(2-2\epsilon,D_d,d)}d^{1-\epsilon} + \|u\|_{C_TD_d}[\fI\xi]_{n,1-\epsilon} + \|u\|_{C_TD_d}d^{-1+\epsilon}
     \end{equation*}
     for any $d\leq d_0$. Now we substitute $d_0 = c_0C^{-\frac{1}{2}}$ and obtain
     \[
        [u]_{C_T(1-\epsilon,D_{d_0},d_0)} \lesssim (c_0^{-1+\epsilon}C^{\frac{1-\epsilon}{2}} + C^{\frac{1-\epsilon}{2}})\|u\|_{C_T(D)}.
     \]
     Thus, we have actually 
     \[
     [u]_{C_T(1-\epsilon,D_{d_0})}\lesssim (c_0^{-1+\epsilon} + 1)C^{\frac{1-\epsilon}{2}}\|u\|_{C_T(D)} + d_0^{-1+\epsilon}\|u\|_{C_T(D)}.
     \]
     This gives the result for $[u]_{C_T(1-\epsilon,D_{d_0})}$. The bounds for $\|\nu\|_{C_T(1-2\epsilon,\fP{r},\tilde{\theta})}, \|U\|_{C_T(\fP{r},\tilde{\theta})}$ and $[U]_{C_T(2-2\epsilon,\fP{r},\tilde{\theta})}$ follow from similar arguments.
 \end{proof}
    
\section{Proof of the main theorem}\label{sec.proof_of_main_thm}
In this section, we will discuss the existence and uniqueness of solutions to \eqref{equ.PAM+square+phi} and conclude the  proof of Theorem \ref{thm.rsbm_compact_support_property}. To achieve both, we have to control space-time regularity rather than  space regularity only, which can be achieved with the help of the following Schauder estimation.

	\begin{lem}\label{lem.Schauder_est.heat}
	   Suppose $u$ is a mild solution to the heat equation
	   \begin{equation}\label{equ.heat}
	       \left\{
	            \begin{array}{cc}
	                 (\partial_t - \Delta)u = f,&  \text{ in }[0,T]\times\RR^2,\\
	                 u(0,\cdot) = u_0,& \text{ on }\{0\}\times \RR^2.
	            \end{array}
	       \right.
	   \end{equation}
	   Let $\theta\in\pmb{\varrho}(\pmb{\omega})$ be a weight such that $\theta(x)\geq 1$. Then we have for $\alpha \in (0,2)$:
	   \begin{equation}
	       \|u\|_{\fL^\alpha(\RR^2,\theta)} \lesssim \|f\|_{C_T(\alpha-2,\RR^2,\theta)} + \|u_0\|_{\alpha,\RR^2,\theta}.
	   \end{equation}
	\end{lem}
	\begin{proof}
	    See proposition 3.6 and lemma 3.10 in \cite{martin2019paracontrolled} for a proof in the discrete case, which leads to the continuous case via a limiting argument. See also lemma 2.10 in \cite{gubinelli2019global} for a slightly different equation. 
	\end{proof}

Apart from the lemma \eqref{lem.Schauder_est.heat}, another difficulty we meet in solving the singular SPDE is the multiplication of distributions. For example, the product of $\xi$ and $u$, which are of regularities $-1-$ and $1-$, is ill-defined without imposing further structural conditions. We use paracontrolled distributions to overcome this issue. 

 \begin{definition}
    For a weight $\theta \in \pmb{\varrho}(\pmb{\omega})$, we say $u$ is \emph{paracontrolled} if $u \in C^{1-\epsilon}(\RR^2,\theta)$ and 
        \[
            u^{\sharp} = u - u\para \fI\xi \in C^{1+2\epsilon}(\RR^2,\theta).
        \]
        For such $u$ we set
        \[
            \fH u = \Delta u + \xi\para u + u\para \xi + u^{\sharp}\odot\xi +C_1(u,\fI\xi,\xi) + u(\fI\xi\odot\xi).
        \]
    where $C_1(u,\fI\xi,\xi) = (u\para \fI\xi)\odot\xi - u(\fI\xi\odot \xi)$. From~\cite{perkowski2021rough} we know that $\fH u$ is a well-defined distribution in $C^{-1-\epsilon}(\RR^2,p(a)\theta)$ for all $a>0$.
\end{definition}

By analysing the equation of singular part $u^\sharp$, we could obtain the solution theory of any regularized version of PAM. Here we will take constants $C_\alpha$ in the assumption $\ref{ass.rsbm}$ and consider the regularized equation 
     
\begin{equation}\label{equ.renormalized.PAM+f}
        \left\{
        \begin{array}{cc}
            \partial_t u^\alpha  = \fH^\alpha u^\alpha + f^\alpha &  \text{ in }\mathbb{R}_+\times \RR^2\\
             u^\alpha(0,\cdot) = u^\alpha_0& \text{ on } \{0\} \times \RR^2
        \end{array}
    \right.
\end{equation}
where $\fH^\alpha = \Delta +\xi_\alpha-C_\alpha$. We also denote $\fT^\alpha$ to be the semi-group generated by the operator $\fH^\alpha$, then solution $u_\alpha$ to the above have a mild solution representation
\[
    u_\alpha = \int_0^t \fT^\alpha_{t-s}f^\alpha_sds + \fT_tu^\alpha_0.
\]

We here give a simple version of the estimations of solutions of equation \eqref{equ.renormalized.PAM+f} and their limit. See \cite{martin2019paracontrolled} for more detail.  

\begin{theorem}\label{thm.PAM.solution.estimation}
    For $\epsilon\in\left(0,\frac{1}{4}\right)$, taking the weight $\theta$ to be $e(l)$ for some $l \in\RR$, if $f^\alpha \in C_T(-1+2\epsilon,\RR^2,e(l)), u_0^\alpha \in C^{1+2\epsilon}(\RR^2,e(l))$ and 
    \[
        f^\alpha \rightarrow f \text{ in }C_T(-1+2\epsilon,\RR^2,e(l)),\qquad u_0^\alpha\rightarrow u_0 \text{ in }C^{1+2\epsilon}(\RR^2,e(l+\cdot))
    \]
    then we have 
    \[
        u^\alpha \rightarrow u = \int_0^t \fT_{t-s}f_sds + \fT_tu_0\text{ in } \fL^{1-\epsilon}(\RR^2,e(l+\cdot)),
    \]
    where $T$ is the limit operator of $T^\alpha$ and 
    \[
        \|u\|_{\fL^{1-\epsilon}(\RR^2,e(l+\cdot))}\lesssim \|u_0\|_{1+2\epsilon,\RR^2,e(l)} + \|f\|_{C_T(-1+2\epsilon,\RR^2,e(l+\cdot))}.
    \]
\end{theorem}

To solve the nonlinear equation (\ref{equ.PAM+square+phi}) we use a mild formulation: Under suitable regularity and integrability conditions we say that $u$ solves (\ref{equ.PAM+square+phi}) if
\[
    u = \mathcal T_t u_0 + \int_0^t \mathcal T_{t-s} \left(-\frac\kappa2 u_s^2 + \phi\right)ds.
\]
Our proof of existence and uniqueness of mild solutions uses similar arguments as Proposition 4.5 in \cite{perkowski2021rough}.

\begin{theorem}
    Let $T > 0$ and $\epsilon\in\left(0,\frac{1}{4}\right)$, and let $l_0 < -T$. For a non-negative function $\varphi_0 \in C^{1+2\epsilon}(\RR^2,e(l_0))$ and a non-negative function $\phi \in C^{-1+2\epsilon}(\RR^2,e(l_0))$, the solution of equation \eqref{equ.PAM+square+phi} exists and is unique in the space $\fL^{1-\epsilon}_T(\RR^2, e(l))$ with $l = l_0+T$. Moreover, the solution is non-negative.
\end{theorem}

\begin{proof}
    We define the map $\fK(\psi) = g$, where $g$ is the solution to 
    \begin{equation}\label{equ.variant_PAM}
        \left\{
        \begin{array}{cc}
             \partial_t g = (\fH - \frac{\kappa}{2}\psi)g + \phi & \text{ in } \RR_+\times \RR^2,  \\
             g(0) = \varphi_0 & \text{ in }\{0\}\times \RR^2. 
        \end{array}
        \right.
    \end{equation}
     We have a similar estimate as in \cite{perkowski2021rough}:
    \begin{equation}\label{equ.est.PAM+square+phi_prior}
    \begin{aligned}
        \|\fK{\psi}&\|_{\fL^{1-\epsilon}_T(\RR^2,e(l_0+\cdot))}\\\lesssim 
        &(\|\varphi_0\|_{1+2\epsilon,\RR^2,e(l_0+\cdot)} + \|\phi\|_{C_T(-1+2\epsilon,\RR^2,e(l_0+\cdot))})e^{C\|\psi\|_{C_T(\RR^2,e(l_0+\cdot))}}.
    \end{aligned}
    \end{equation}
    Indeed, starting with an estimate of the uniform norm, we have 
    \begin{equation*}
        \begin{aligned}
            \|g\|_{C_T(\RR^2,e(l_0+\cdot))} \lesssim& \|\varphi_0\|_{1+2\epsilon,\RR^2,e(l_0+\cdot)} + \|\phi\|_{C_T(-1+2\epsilon,\RR^2,e(l_0+\cdot))} \\
            &+ \frac{\kappa}{2}\sup_{t\in[0,T]}\left\{\int_0^t\|\fT_{t-s}\psi_sg_s\|_{\RR^2,e(l_0+t)}ds\right\}.
        \end{aligned}
    \end{equation*}
    If $l_0 < -T$, we are able to estimate 
    \[
        \|\fT_{t-s}\psi_sg_s\|_{\RR^2,e(l_0+t)} \lesssim \|\psi\|_{\RR^2,e(l_0+\cdot)}\|g\|_{\RR^2,e(l_0+\cdot)}.
    \]
    Hence, Gronwall's inequality gives 
    \[
         \begin{aligned}
        \|g&\|_{C_T(\RR^2,e(l_0+\cdot))} \lesssim (\|\varphi_0\|_{1+2\epsilon,\RR^2,e(l_0+\cdot)} + \|\phi\|_{C_T(-1+2\epsilon,\RR^2,e(l_0+\cdot))})e^{C\|\psi\|_{C_T(\RR^2,e(l_0+\cdot))}}.
    \end{aligned}
    \]
    Use again the estimate from theorem \ref{thm.PAM.solution.estimation}, we now obtain 
    \begin{equation*}
        \begin{aligned}
            \|g\|_{\fL^{1-\epsilon}_T(\RR^2,e(l_0+\cdot))} \lesssim& \|\varphi_0\|_{1+2\epsilon,\RR^2,e(l_0+\cdot)} + \|\phi-\frac{\kappa}{2}\psi g\|_{C_T(-1+2\epsilon,\RR^2,e(l_0+\cdot))},
        \end{aligned}
    \end{equation*}
    which gives equation \eqref{equ.est.PAM+square+phi_prior}.
    Furthermore, the maximum principle and the comparison principle \cite{cannizzaro2017malliavin, perkowski2020spdes} show that $T_t\varphi_0 + \int_0^tT_{t-s}\phi ds\geq \fK(\psi)(t) \geq0$. Thus, we start with $\varphi^0 = T_t\varphi_0 + \int_0^tT_{t-s}\phi ds$ and then let $\varphi^m = \fK{\varphi^{m-1}}$. We  have the bounds 
    \[
        \|\fK{\psi}\|_{\fL^{1-\epsilon}(\RR^2,e(l_0+\cdot))}\lesssim C(\varphi_0,\phi)e^{C\|T_t\varphi_0 + \int_0^tT_{t-s}\phi ds\|_{C_T(\RR^2,e(l_0+\cdot))}}.
    \]
    To show the existence of the fixed point, we are going to use Schauder's fixed point theorem. To start with, we consider the convex set
    \[
    \mathcal{E} := \{\fK(\psi):\psi \in \fL_T^{1-\epsilon}\}.
    \]
    Then $\fK(\mathcal{E}) \subset \mathcal{E}$ and $\fK(\mathcal{E})$ is a pre-compact set in $\fL_T^{1-\epsilon}$. It remains shown $\mathcal{E}$ and $\fK(\mathcal{E})$ are closed. Taking $\psi_1,\psi_2$ and consider $g_1 = \fK(\psi_1),g_2 = \fK(\psi_2)$, we have the equation 
    \[
        \partial_t(g_1 - g_2) = \left(\fH - \frac{\kappa}{2}\right)(g_1 - g_2) - \frac{\kappa}{2}(\psi_1 - \psi_2)g_2.
    \]
    Then it is obvious that $\mathcal{E}$ and $\fK(\mathcal{E})$ are closed. This then gives us the existence of the fixed point. Hence the existence of the solution to  equation \eqref{equ.PAM+square+phi}.  Uniqueness follows from the same arguments as Proposition 4.5 in \cite{perkowski2021rough}.
\end{proof}

\begin{remark}
    Our existence proof for equation \eqref{equ.PAM+square+phi} only works for $l_0 < -T$. But for proving uniqueness we have to consider the difference of two solutions $g=\varphi^1-\varphi^2$, which solves  equation \eqref{equ.variant_PAM} with initial condition $0$ and $\phi =0$, and for $\psi = \varphi^1 + \varphi^2$. If $\varphi^i \in \fL_T^{1-\epsilon}(\RR^2, p(a))$ for sufficiently small $a>0$, then we can use the same arguments as in \cite{hairer2015simple, martin2019paracontrolled} to show that $g=0$: By giving up a bit of regularity or by introducing a blow up at time 0 we are able to kill the polynomial weight by the Schauder estimates for the heat equation. Thus we are  able to prove the uniqueness to the equation (\ref{equ.PAM+square+phi}) in $\fL_T^{1-\epsilon}(\RR^2, p(a))$ if $a>0$ is sufficiently small.
    
\end{remark}

	Together with corollary \ref{cor.local_est_weighted.renormalized.PAM+square} and the above estimations, we can estimate space-time weighted norms of solution of equation $(\ref{equ.renormalized.PAM+square})$. For each $r > 2$, we consider an auxiliary function $\eta_r$ with
 \begin{equation}\label{def.eta}
    \eta_r = \left\{
        \begin{aligned}
             &1, \quad\quad \text{ in } \fP{r-2}, \\
             &0, \qquad \text{ out side of  } \fP{r-1},\\
             &[0,1], \quad \text{ elsewhere. }
        \end{aligned}
    \right.
\end{equation} 
	\begin{lem}\label{lem.local_est_weighted_spacetime.renormalized.PAM+square}
	    Let $r > 2$ and consider a weight $\theta \in \pmb{\varrho}(\pmb{\omega})$ such that $\theta(x) = \theta(|x|) \geq 1$ and $\theta(|x|)$ is increasing in $|x|$. If $u$ solves equation $(\ref{equ.renormalized.PAM+square})$ in $[0,T]\times \fP{r}$, then we have for $\tilde{\theta} = \theta^{1+\frac{4}{1-\epsilon}}$:
        \begin{equation}
            \|u\eta_{r}\|_{\fL^{1-\epsilon}(\RR^2,\tilde{\theta})} \lesssim 1 + \sum_{\tau = \xi,\fI\xi\xi} [\tau]_{r,|\tau|,\theta}^{1+\frac{4}{n_\tau(1-\epsilon)}}.
        \end{equation}
	\end{lem}
	\begin{proof}
	    The equation of $u\eta_r$ is given by 
	    \[
	        (\partial_t-\Delta)(u\eta_r) = (\xi-C)u\eta_r - \frac{\kappa}{2}\eta_ru^2 - 2\nabla( u\cdot\nabla\eta_r) + u\Delta\eta_r.
	    \]
	    By Lemma \ref{lem.est.irregular_term} and step 2 in the proof of Theorem \ref{thm.local_est.renormalized.PAM+square}, we get a bound for the H\"older norm,
	    \begin{equation}\label{quantity.upper_bound}
	        \|u(\xi - C)\|_{C_T(-1-\epsilon,\fP{r-1},\tilde{\theta})} \lesssim 1 + \sum_{\tau = \xi,\fI\xi\xi} [\tau]_{r,|\tau|,\theta}^{1+\frac{4}{n_\tau(1-\epsilon)}}.
	    \end{equation}
	    Since $\eta_r$ is supported on $\fP{r-1}$, we know $(\xi-C)u\eta_r \in C_T(-1-\epsilon,\RR^2,\tilde{\theta})$ and has the same bound $(\ref{quantity.upper_bound})$ up to a constant. It is directly from Corollary \ref{cor.est_all} that we have
	    \begin{eqnarray*}
	        &&\|\eta_ru^2\|_{C_T(\RR^2,\tilde{\theta})}, \|\nabla( u\cdot\nabla\eta_r)\|_{C_T(-\epsilon,\RR^2,\tilde{\theta})},\|u\Delta\eta_r\|_{C_T(1-\epsilon,\RR^2,\tilde{\theta})} \\&\lesssim& 	        1 + \sum_{\tau = \xi,\fI\xi\xi} [\tau]_{r,|\tau|,\theta}^{\frac{4}{n_\tau(1-\epsilon)}}.
	    \end{eqnarray*}
	    Thus, the desired result follows from the Schauder estimate Lemma \ref{lem.Schauder_est.heat}.
	\end{proof}
    By approximating equation $(\ref{equ.PAM+square+phi})$ with regularized equations, we  see that the function $u\eta_r$ in the above theorem has a mild solution representation:
    \begin{lem}\label{lem.mild_rep}
        Let $l_0 <-T$ and $\epsilon \in \left(0,\frac{1}{4}\right)$. Suppose $u$ is a solution of equation $(\ref{equ.PAM+square+phi})$ with non-negative $\varphi_0 \in C^{1+2\epsilon}(\RR^2,e(l_0))$ and non-negative $\phi \in C^{-1+2\epsilon}(\RR^2,e(l_0))$. Let $\eta\in C_c^\infty(\RR^2)$ be a compactly supported smooth function. Then
        \begin{equation}
        \begin{aligned}
            (u\eta)(t,x) =& \int_0^t T_{t-s}\left( - \frac{\kappa}{2}\eta u^2 + \phi\eta- 2\nabla (u\cdot\nabla\eta) + u\Delta\eta\right)_s(x)ds \\
            &+ T_t(\varphi_0\eta)(x).
            \end{aligned}
        \end{equation} 
    \end{lem}
\begin{proof}
    Since $u$ is a solution of equation $\eqref{equ.PAM+square+phi}$, it is a limit of $u_\alpha$ with
\begin{equation}\label{equ.renormalized.PAM+square+phi}
        \left\{
        \begin{array}{cc}
            \partial_t u^\alpha  = \fH^\alpha u^\alpha -\frac{\kappa}{2}(u^\alpha)^2 + \phi^\alpha, &  \text{ in }\mathbb{R}_+\times \RR^2,\\
             u^\alpha(0,\cdot) = \varphi^\alpha_0,& \text{ on } \{0\} \times \RR^2,
        \end{array}
    \right.
\end{equation}
as $\alpha\rightarrow 0$, where $\varphi_0^\alpha$ converges to $\varphi_0$ in $C^{1+2\epsilon}(\RR^2,e(l_0))$ and $\phi^\alpha$ converges to $\phi$ in $C^{-1+2\epsilon}(\RR^2,e(l_0))$. Since the regularized equation can be solved classically, the function $u^\alpha\eta$ solves
\begin{equation}
        \left\{
        \begin{array}{ll}
            \partial_t (u^\alpha\eta)  = \fH^\alpha (u^\alpha\eta) -\frac{\kappa}{2}(u^\alpha)^2\eta + \phi\eta - 2\nabla (u^\alpha\cdot\nabla\eta) + u^\alpha\Delta\eta, &  \text{ in }\mathbb{R}_+\times \RR^2,\\
             u^\alpha(0,\cdot)\eta = \varphi^\alpha_0\eta,& \text{ on } \{0\}\times\RR^2.
        \end{array}
    \right.
\end{equation}
Thus, we have 
\[
    u_\alpha\eta = \int_0^t T^\alpha_{t-s}f^\alpha_sds + T_t(\varphi^\alpha_0\eta),
\]
where $f^\alpha = -\frac{\kappa}{2}(u^\alpha)^2\eta + \phi\eta - 2\nabla (u^\alpha\cdot\nabla\eta) + u^\alpha\Delta\eta$. By checking the regularity of $u^\alpha$, we know $f^\alpha$ is of regularity $-1+2\epsilon$ with weight $e(2(l_0+\cdot))$ and $\varphi^\alpha_0\eta$ is of regularity $1+2\epsilon$ with weight $e(l_0)$. Thus we can let $\alpha \rightarrow 0$ and obtain the result.
\end{proof}

We are finally ready to give the proof of our main result, Theorem \ref{thm.rsbm_compact_support_property}:
\begin{proof}[Proof of Theorem \ref{thm.rsbm_compact_support_property}]
As in Section \ref{sec.rsbm}, we consider solutions $\varphi_n^m$ of equation $(\ref{equ.PAM+square+phi})$ with $\phi_n^m$ defined in Section \ref{sec.rsbm}. By lemma \ref{lem.mild_rep} we know that
\begin{equation}
    \varphi_n^m\eta_n  = \int_0^t T_{t-s}(- (\varphi_n^m)^2\eta_n - 2\nabla(\varphi_n^m\cdot\nabla\eta_n) +\varphi_n^m\Delta\eta_n)_sds.
\end{equation}

By lemma \ref{lem.local_est_weighted_spacetime.renormalized.PAM+square} we have for any $a > 0$
\begin{equation}
    \|\varphi_n^m\eta_n\|_{\fL^{1-\epsilon}(\RR^2,p(a))} \lesssim 1.
\end{equation}
By compact embedding $\varphi_n^m\eta_n$ converges (along a sub-sequence) in $\fL^{1-\epsilon'}(\RR^2,p(a'))$ to some $\varphi_n\eta_n$  as $m$ tends to infinity, where $\epsilon'>\epsilon$ and $a'>a$ correspond to an arbitrarily small loss of regularity/weight. From the convergence we get a representation of $\varphi_n\eta_n$:
\begin{equation}
    \varphi_n\eta_n  = \int_0^t T_{t-s}(- (\varphi_n)^2\eta_n - 2\nabla(\varphi_n\cdot\nabla\eta_n) +\varphi_n\Delta\eta_n)_sds.
\end{equation}
The same argument applied to $\varphi_n\eta_n$ shows that $\varphi:=\lim_{n\rightarrow}\varphi_n\eta_n$ actually is a mild solution of the equation (\ref{equ.PAM+square}) with $0$ initial condition. Hence $\varphi$ is identically $0$ and we have, for any compactly supported finite measure $\mu(0)$, that
\[
\lim_{n\rightarrow\infty}\lim_{m\rightarrow\infty}\exp(-\langle \mu(0),\varphi_n^m(t)\rangle) = \lim_{n\rightarrow\infty}\lim_{m\rightarrow\infty}\exp(-\langle\mu(0),\varphi_n^m\eta_n(t)\rangle) = 1.
\]
Thus, the compact support property of rough super Brownian motion holds. 
\end{proof}

\appendix
\section{Proofs of some technical lemmas}\label{app}
Here we give proofs of some technical lemmas. Let first recall the lemma \ref{lem.general_equ}.
	\begin{lem*}
		Let $T>0$ and let $u\in C^\infty$ solve
  	     \begin{equation}
		  \left\{
			\begin{array}{ll}
				(\partial_t - \Delta)u = - u^2 + g & \text{ in } [0,T]\times \fP{n}, \\
				u = 0 & \text{ in } \{0\}\times \fP{n}, \\
				u\geq 0,
			\end{array}
		  \right.
	   \end{equation}
	where $g$ is a smooth and bounded function. Then the following point-wise bound on $u$ holds for all $z = (t,x) \in [0,T)\times \fP{n}$:
		\begin{equation}
			u(t,x) \leq 28\cdot \max\left\{\frac{1}{\min\{(n-x_i)^2,(n+x_i)^2, i = 1,2\}},\sqrt{\|g\|_{C_T \fP{n}}}\right\}.
		\end{equation}
	\end{lem*}
	\begin{proof}\label{pr.lem.general_equ}
		We consider a $C^2$ function $\eta$ on $\fP{n}$ such that $\eta = 0$ on $\partial \fP{n}$ and $\eta$ is strictly positive on the interior $(-n,n)^2$. Then $u\eta$ is $0$ on the parabolic boundary of $\RR_+\times \fP{n}$ and non-negative in $\RR_+\times \fP{n}$. 

		Let $\fP{n}^t = [0,t]\times \fP{n}$
        and let $z_t$ be the maximum point of $u\eta$ in the region $\fP{n}^t$. If $z_t \in \partial_p\fP{n}^t$, then $u\eta = 0$ on $\Omega_t$ and $u=0$ in $\mathbb{R}_+\times (-n,n)^2$. Here $\partial_p$ means parabolic boundary. The results holds in this case. Now suppose $z_t \notin \partial_p\fP{n}^t$ and $(u\eta)(z_t) > 0$. Then we have $\nabla(u\eta)(z_t) = 0,\Delta(u\eta)(z_t) \leq 0$ and $\partial_t(u\eta)(z_t) \geq 0$. Hence, we obtain
		\[
			\nabla u(z_t) = - \frac{u\nabla\eta}{\eta}(z_t).
		\]
		Moreover, we have at $z_t$:
		\begin{eqnarray*}
			0 &\leq& (\partial_t - \Delta)(u\eta)(z_t) \\
			&=& \big(\eta(\partial_t - \Delta)u + u(\partial_t-\Delta)\eta -2\nabla u\cdot \nabla \eta\big)(z_t) \\
			&=& \left(\eta (-u^2 + g) + u\left(-\Delta\eta + 2\frac{|\nabla\eta|^2}{\eta}\right)\right)(z_t),
		\end{eqnarray*}
		which yields
		\[
			u(z_t) \leq \frac{g}{u}(z_t) - \frac{ \Delta\eta}{\eta}(z_t) + 2\frac{|\nabla\eta|^2}{\eta^2}(z_t).
		\]
		Denoting $\tilde \eta:= \frac{1}{\eta}$, we obtain the inequality 
		\[
			u(z_t) \leq  \frac{g}{u}(z_t) + \frac{\Delta\tilde \eta}{\tilde \eta}(z_t).
		\]
        Consider now $\tilde \eta = \sum_{i=1}^2\left(\frac{1}{(n-x_i)^2} + \frac{1}{(n+x_i)^2}\right) + \sqrt{\|g\|_{[0,T]\times \mathcal P_n}}$, for which $\eta = \frac{1}{\tilde \eta}$ satisfies the assumptions made above. Moreover, we have $\Delta \tilde \eta \leq 6\tilde \eta^2$ and $g \leq \tilde \eta^2$, and thus
        \[
            \frac{u}{\tilde \eta}(z_t) \leq \frac{\tilde \eta^2}{u\tilde \eta}(z_t) + 6 = \frac{\tilde \eta}{u}(z_t) + 6.
        \]
        The inequality $x \le \frac1x + 6$ for $x \ge 0$ can only be satisfied if $x \le 7$. Otherwise we would get $\frac1x < 7$ and then $x < \frac17 + 6 < 7$, a contradiction. Therefore, using that $z_t$ is the maximum of $\frac{u}{\tilde \eta}$ on $\fP{n}^t$, we get
        \[
            u \le 7 \cdot \tilde \eta \le 28 \cdot \max\left\{\frac{1}{\min\{(n-x_i)^2,(n+x_i)^2, i = 1,2\}},\sqrt{\|g\|_{C_T \fP{n}}}\right\},
        \]
        as claimed.
	\end{proof}

	In order to give a proof of Lemma \ref{lem.Schauder_est.two_variable}, we will need an interior gradient estimate for the heat equation. 
	
	\begin{lem}\label{lem.gradient_est.heat_interior}
	   Suppose $u$ satisfies the equation 
	\begin{equation}\label{equ.heat_0}
		\left\{
			\begin{array}{ll}
				(\partial_t - \Delta)u = 0 & \text{ in } \RR_+\times B(x,L) \\
				u = 0 & \text{ on } \{0\}\times B(x,L) 
			\end{array}
		\right.
	\end{equation}
	for some $x \in \RR^d$. Let $R < \frac{L}{2}$ and $T>0$. Then we have the interior gradient estimate for $T > 0$ for any $f \in C(B(x,L))$,
	\[
	   \|\nabla  u\|_{C_TB(x,R)} \lesssim \frac{1}{L}\|u - f\|_{C_TB(x,L)}
	\]
	uniformly in $T$ and $f$.
	\end{lem}
    We suspect that this is well-known, but were unable to find a reference.
	\begin{proof}
	    We first take $f=0$ and argue later how to treat general $f$. By scaling, it suffices to consider the case $L=1$ and $R=\frac12$. Classical inner regularity theory for the heat equation gives (for $T \ge \frac14$)
        \[
            \sup_{t \in [\frac14,T]} \|\nabla u(t,\cdot)\|_{B(x,\frac12)} \lesssim \|u\|_{C_T B(x,L)},
        \]
        see for example Theorem~8.4.4 in \cite{krylov1996lectures} or Theorem~IV.4.8 in \cite{lieberman1996second}. These estimates are independent of the initial condition, and we leverage the initial condition $u(0)=0$ to extend the estimate to $t \in [0,\frac14]$.

        In our argument we will consider two different spatial scales, and for that purpose it is convenient to reintroduce $R$ and $L$ in the notation.
	 Let $\eta \in C_c^\infty(\RR^d)$ be such that $\eta = 1$ on $B(x,R)$, $\eta = 0$ outside of $B(x,\frac{R+L}{2})$ and $\|\nabla\eta\|_\infty \leq \frac{4}{L-R}$, $\|\Delta \eta\|_\infty \leq \frac{8}{(L-R)^2}$. Let $w = u\eta$, so that
	    	\begin{equation}
		\left\{
			\begin{array}{ll}
				(\partial_t - \Delta)w = -2\nabla(u\nabla\eta) + u\Delta\eta, \\
				w(0,\cdot) = 0, 
			\end{array}
		\right.
	\end{equation}
	on $\RR_+ \times \RR^d$, and therefore
	\[
	    w(t) = \int_0^t K_{t-s}*(-2\nabla(u(s)\nabla\eta) + u(s)\Delta\eta) ds,
	\]
	where $K$ is the heat kernel. Consider the difference operator $D_h f(x) = f(x+h) - f(x)$, and note that by interpolation for  $\alpha=\frac34$ (the following argument works for any $\alpha \in (\frac12,1)$):  
    \[
        \| D_h \nabla K_{t-s}\|_{L^1} \lesssim ( h \|\nabla^2 K_{t-s}\|_{L^1})^\alpha ( \|\nabla K_{t-s}\|_{L^1})^{1-\alpha} \lesssim h^\alpha (t-s)^{-\alpha} (t-s)^{-\frac{1-\alpha}{2}},
    \]
    which yields
    \begin{align*}
        \left|\int_0^t (D_hK_{t-s})*(-2\nabla(u(s)\nabla\eta))ds \right| & = \left|2\int_0^t(D_h\nabla K_{t-s})*(u\nabla\eta)ds\right| \\
        &\lesssim \frac{|h|^\alpha t^{\frac{1-\alpha}{2}}}{L-R}\|u\|_{C_tB(x,L)}.
    \end{align*}
    Similarly, we have 
    \[
        \left|\int_0^t(D_hK_{t-s})*(u(s)\Delta\eta)\right| \lesssim \frac{|h|^\alpha t^{1-\frac{\alpha}{2}}}{(L-R)^2}\|u\|_{C_tB(x,L)},
    \]
    and therefore
    \[
        \|D_h w(t)\|_\infty \lesssim |h|^\alpha t^{\frac{1-\alpha}{2}}\|u\|_{C_tB(x,L)}\left(\frac{1}{L-R}+\frac{t^{\frac{1}{2}}}{(L-R)^2}\right).
    \]
    Recalling that $w=u\eta$ and that $\eta|_{B(x,R)} \equiv 1$, this means
    \[
        [u(t,\cdot)]_{\alpha,B(x,R)} \lesssim t^{\frac{1-\alpha}{2}}\left(\frac{1}{L-R}+\frac{t^{\frac{1}{2}}}{(L-R)^2}\right) \|u\|_{C_tB(x,L)}.
    \]
    Since we had to take $\alpha<1$ to keep the singularity $(t-s)^{-\frac{1+\alpha}{2}}$ integrable, this is not yet sufficient. Therefore, we iterate the argument. Let now $R_1 \in (R,L)$ and let $|h| \in (0,L-R_1)$. Then $\frac{D_h u}{|h|^\alpha}$ also solves the heat equation (\ref{equ.heat_0}) on $B(x,R_1)$, so by the previous argument
    \[
        \left[\frac{D_h u}{|h|^\alpha}(t,\cdot)\right]_{\alpha,B(x,R)}\lesssim t^{\frac{1-\alpha}{2}}\left(\frac{1}{R_1-R} +\frac{t^{\frac{1}{2}}}{(R_1-R)^2}\right)\left\|\frac{D_h u}{|h|^\alpha}\right\|_{C_tB(x,R_1)}.
    \]  
    Now we take $L = 1$, $R_1 = \frac{3}{4}$ and $R = \frac{1}{2}$, and combine the above inequalities to obtain for any $|h| < L-R_1 = \frac{1}{4}$, 
    \[
         \left[\frac{D_h u}{|h|^\alpha}(t,\cdot)\right]_{\alpha,B(x,\frac{1}{2})}\lesssim t^{1-\alpha}(1 +\sqrt{t})^2\|u\|_{C_tB(x,1)}.
    \]
    By Lemma 5.6 in \cite{roberts1995fully}, we have then (since $\alpha = \frac34 > \frac12$)
    \[
        \|\nabla u(t,\cdot)\|_{B(x,\frac{1}{2})} \lesssim t^{1-\alpha}(1+t)\|u\|_{C_tB(x,1)}
    \]
    Let $T>1$ as an upper time horizon. Now when $t \leq \frac{1}{4}$, we have 
    \begin{equation}\label{equ.gradient_est.heat}
        \|\nabla u(t,\cdot)\|_{B(x,\frac{1}{2})} \lesssim \|u\|_{C_TB(x,1)},
    \end{equation}
    which is the desired result for $f=0$ (recall that we only have to consider $R=\frac12$, $L=1$, and $t \le \frac14$).
    
    For general continuous function $f$, we simply use that $u(0) = 0$ and that $f$ does not depend on time to estimate
    \begin{align*}
        \| u \|_{C_T B(x,L)} & \le \| u - f\|_{C_T B(x,L)} + \|f\|_{B(x,L)} \\
        & = \| u - f\|_{C_T B(x,L)} + \| u(0,\cdot) - f\|_{B(x,L)} \\
        & \le 2 \|u-f\|_{C_T B(x,L)}.
    \end{align*}
	\end{proof}
 
Now we can give the proof of Lemma \ref{lem.Schauder_est.two_variable}. 
\begin{proof}[Proof of Lemma \ref{lem.Schauder_est.two_variable}]\label{pr.lem.Schauder_est.two_variable}
    
	   The proof is very similar to that of Lemma 2.11 in the paper \cite{moinat2020space}. The only change is that we replace the parabolically shrunk region by the spatially shrunk cylinder. We include the proof for completeness.
	   
	   \quad \textit{Step 1.} We claim that for all base points $x$ and scales $\delta,R$ and $L$ with $R\leq \frac{L}{2}$ and such that $B(x,L) \subset D$, it holds that for $0 \leq t < T = kR^2, k \geq 1$
	   \begin{equation}\label{three variable.step1}
	       \begin{aligned}
	           \inf_{l} \|U_\delta(t,x,\cdot) - l\|_{C_TB(x,R)} \lesssim& \frac{kR^2}{L^2}\inf_{l}\|U_\delta(t,x,\cdot) - l\|_{C_TB(x,L)} \\
	        &+ kL^2M_{\{x\},L}^{(1)}\sum_{\beta\in A}\delta^{\beta-2}L^{\kappa-\beta},
	       \end{aligned}
	   \end{equation}
	   where the infimum runs over all spatial affine functions $l(y) = C\cdot (y-x) + c$. To prove this, we define a decomposition $U_\delta = u_{>}+u_{<}$ where $u_{>}$ is the solution to 
	   \[
	    (\partial_t - \Delta)u_>(t,y) = \1_{B^T(x,L)}(\partial_t-\Delta_y)U_\delta(t,x,y),
	   \]
	   with Dirichlet boundary conditions. By standard estimates for the heat equation \cite{moinat2020space} and assumptions in this lemma,
        \begin{eqnarray*}
            \|u_>\|_{C_TB(x,L)} &\lesssim& L^2\|(\partial_t-\Delta_y) U_\delta(t,x,y))\|_{y\in B(x,L)} \\
            &\leq& L^2M_{\{x\},L}^{(1)}\sum_{\beta\in A}\delta^{\beta-2}L^{\kappa-\beta}.
        \end{eqnarray*}
        Now $(\partial_t - \Delta)u_< = 0$ on $B^T(x,L)$ and $u_<(0,\cdot) = 0$ on $B(x,L)$. By Lemma \ref{lem.gradient_est.heat_interior}, we know directly for $\partial \in \{\partial_t,\partial_i\partial_j\}$ a differential operator of order 1 in time and 2 in space, 
        \[
            \|\partial u_<\|_{C_TB(x,R)} \leq L^{-2}\|u_< - l_>\|_{C_TB(x,R)},
        \]
        for any affine function $l_>$ with $R\leq \frac{L}{2}$. Let $l_<(y) = u_<(T,x) + \nabla u_<(T,x)(y-x)$, the Taylor's formula show that 
        \[
            \|u_< - l_<\|_{C_TB(x,R)} \leq kR^2\|Du_<\|_{C_TB(x,R)} \lesssim \frac{kR^2}{L^2}\|u_< - l_>\|_{C_TB(x,L)}.
        \]
        Thus,
        \begin{eqnarray*}
            &&\|U_\delta(t,x,\cdot) - l_<\|_{C_TB(x,R)} \\
            &\leq& \|u_< - l_<\|_{C_TB(x,R)} + \|u_>\|_{C_TB(x,R)} \\
            &\lesssim& \frac{kR^2}{L^2}\|U_\delta(t,x,\cdot) - l_>\|_{C_TB(x,L)} + k\|u_>\|_{C_TB(x,L)}\\
            &\leq& \frac{kR^2}{L^2}\|U_\delta(t,x,\cdot) - l_>\|_{C_TB(x,L)} + kL^2M_{\{x\},L}^{(1)}\sum_{\beta\in A}\delta^{\beta-2}L^{\kappa-\beta},
        \end{eqnarray*}
        which implies equation (\ref{three variable.step1}).
        
        \textit{Step 2.} We claim that for all base point $x$ and scales $\delta$ and $L$, it holds that 
        \begin{equation}\label{three variable.step2}
            \|U_\delta(t,x,\cdot) - U(t,x,\cdot)\|_{C_TB(x,R)} \leq M_{\{x\},R,\delta}^{(2)}\sum_{\beta\in A}R^\beta\delta^{\kappa-\beta} + \delta^\kappa[U]_{C_T(\kappa,B(x,R),\delta)}.
        \end{equation}
        Indeed, by symmetry of $\Psi$, 
        \begin{eqnarray*}
            &&|U_\delta(t,x,y) - U(t,x,y)| \\
            &=& \left|\int \Psi^\delta(y-y_1)(U(t,x,y_1) - U(t,x,y))dy_1\right|\\
            &=& \inf_{\nu(t,y)}\left|\int\Psi^\delta(y-y_1)(U(t,x,y_1) - U(t,x,y) - U(t,y,y_1))\right.\\
            && +\left.\int\Psi^\delta(y-y_1)(U(t,y,y_1) - \nu(t,y)(y_1-y))dy_1\right|\\
            &\leq& M_{\{x\},R,\delta}^{(2)}\sum_{\beta\in A}d(x,y)^\beta\int\Psi^\delta(y-y_1)d(y,y_1)^{\kappa-\beta}dy_1 \\
            &&+(\sup_{y\in B(x,R)}\inf_{\nu(t,y)}\sup_{y_1\in B(y,\delta)} d(y,y_1)^{-\kappa}|U(t,y,y_1)-\nu(t,y)(y_1-y)|)\\
            &&\times \int\Psi^\delta(y-y_1)d(y,y_1)^\kappa dy_1.
        \end{eqnarray*}
	
    \textit{Step 3.} We prove for small enough $\epsilon(T)$, we have
    \begin{equation}\label{three variable.step3}
        \begin{aligned}
            &\sup_{R\leq\frac{\epsilon d}{2}}R^{-\kappa}\inf_l\|U(t,x,\cdot) - l\|_{C_TB(x,R)} \\
        \lesssim_{T,\epsilon}& \sum_{\beta\in A}\left(M_{\{x\},\frac{d}{2}}^{(1)}\epsilon^{-4+2\beta-\kappa} + M_{\{x\},\frac{d}{2},\frac{\epsilon^2d}{2}}^{(2)}\epsilon^{\kappa-\beta}(1+\epsilon^{\kappa-\beta})\right)\\
        &+\epsilon^{2-2\kappa}\frac{d^{-\kappa}}{2^{-\kappa}}\|U(t,x,\cdot)\|_{C_TB(x,\frac{d}{2}(1+\epsilon^2))} +(\epsilon^\kappa+\epsilon^{2+\kappa})[U]_{C_T(\kappa,B(x,\frac{\epsilon d}{2}),\frac{\epsilon^2d}{2})}.
        \end{aligned}
    \end{equation}
    Multiplying equation (\ref{three variable.step1}) by $R^{-\kappa}$ and fixing the length ratios $R = \epsilon L = \epsilon^{-1}\delta$ for some $\epsilon\leq \frac{1}{2}$ to be fixed below, we get for any point $x \in D_d$ and length $L\leq \frac{d}{2}$, 
    \begin{eqnarray*}
        &&R^{-\kappa}\inf_l\|U_\delta(t,x,\cdot)-l\|_{C_TB(x,R)}\\
        &\lesssim&k\epsilon^{2-\kappa} L^{-\kappa}\inf_{l}\|U_\delta(t,x,\cdot) - l\|_{C_TB(x,L)} + kM_{D_d,L}^{(1)}\sum_{\beta\in A}\epsilon^{-4+2\beta-\kappa}.
    \end{eqnarray*}
    Taking the supermum over $L\leq \frac{d}{2}$ while keeping the ratios $R = \epsilon L = \epsilon^{-1}\delta$ fixed, we get 
    \begin{eqnarray*}
        &&\sup_{R\leq\frac{\epsilon d}{2}} R^{-\kappa} \inf_l\|U_\delta(t,x,\cdot) - l\|_{C_TB(x,R)} \\
        &\lesssim& k\epsilon^{2-\kappa}\sup_{L\leq \frac{d}{2}}L^{-\kappa} \inf_l \|U_\delta(t,x,\cdot) - l\|_{C_TB(x,L)} + k\sup_{L\leq \frac{d}{2}}M_{D_d,L}^{(1)}\sum_{\beta \in A}\epsilon^{-4+2\beta-\kappa} \\
        &\leq&  k\epsilon^{2-\kappa}\sup_{L\leq \frac{\epsilon d}{2}}L^{-\kappa} \inf_l \|U_\delta(t,x,\cdot) - l\|_{C_TB(x,L)} + k\sup_{L\leq \frac{d}{2}}M_{D_d,L}^{(1)}\sum_{\beta \in A}\epsilon^{-4+2\beta-\kappa} \\
        &&+ k\epsilon^{2-\kappa}\sup_{\frac{\epsilon d}{2} < L \leq \frac{d}{2}}L^{-\kappa}\inf_l\|U_\delta(t,x,\cdot) - l\|_{C_TB(x,L)}.
    \end{eqnarray*}
    The last term is bounded by 
    \[
        k\epsilon^{2-\kappa}\left(\frac{\epsilon d}{2}\right)^{-\kappa}\|U_\delta(t,x,\cdot)\|_{C_TB(x,\frac{d}{2})} \leq k\epsilon^{2-2\kappa}\frac{2^\kappa}{d^\kappa}|U(t,x,\cdot)|_{C_TB(x,\frac{d}{2}(1+\epsilon^2))}.
    \]
    Hence we have 
    \begin{eqnarray*}
        &&\sup_{R\leq\frac{\epsilon d}{2}} R^{-\kappa} \inf_l\|U_\delta(t,x,\cdot) - l\|_{C_TB(x,R)} \\
        &\lesssim& k\epsilon^{2-\kappa}\sup_{L\leq \frac{\epsilon d}{2}}L^{-\kappa} \inf_l \|U_\delta(t,x,\cdot) - l\|_{C_TB(x,L)} + kM_{D_d,\frac{d}{2}}^{(1)}\sum_{\beta \in A}\epsilon^{-4+2\beta-\kappa} \\
        &&+ k\epsilon^{2-2\kappa}\frac{2^\kappa}{d^\kappa}\|U(t,x,\cdot)\|_{C_TB(x,\frac{d}{2}(1+\epsilon^2))}.
    \end{eqnarray*}
    Applying equation $(\ref{three variable.step2})$, we obtain
    \begin{eqnarray*}
        &&\sup_{R\leq \frac{\epsilon d}{2}} R^{-\kappa}\inf_l\|U(t,x,\cdot) - l\|_{C_TB(x,R)} \\
        &\lesssim& \sum_{\beta\in A}\left(kM_{D_d,\frac{d}{2}}\epsilon^{4-2\beta-\kappa} + M_{\{x\},\frac{\epsilon d}{2},\frac{\epsilon^2 d}{2}}\epsilon^{\kappa-\beta}\right) + \epsilon^\kappa[U]_{C_T(\kappa,B(x,\frac{\epsilon d}{2}),\frac{\epsilon^2 d}{2})}\\
        &&+ k\epsilon^{2-2\kappa}\frac{2^\kappa}{d^{\kappa}}\|U(t,x,\cdot)\|_{C_TB(x,\frac{d}{2}(1+\epsilon^2))}\\
        &&+k\epsilon^{2-\kappa}\sup_{L\leq\frac{\epsilon d}{2}}L^{-\kappa}\inf_l\|U_\delta(t,x,\cdot) - l\|_{C_TB(x,L)}\\
        &\lesssim_{T}& \sum_{\beta\in A}\left(M_{\{x\},\frac{d}{2}}^{(1)}\epsilon^{-4+2\beta-\kappa} + M_{\{x\},\frac{d}{2},\frac{\epsilon^2d}{2}}^{(2)}\epsilon^{\kappa-\beta}(1+\epsilon^{\kappa-\beta})\right)\\
        &&+\epsilon^{2-2\kappa}\frac{d^{-\kappa}}{2^{-\kappa}}\|U(t,x,\cdot)\|_{C_TB(x,\frac{d}{2}(1+\epsilon^2))} +(\epsilon^\kappa+\epsilon^{2+\kappa})[U]_{C_T(\kappa,B(x,\frac{\epsilon d}{2}),\frac{\epsilon^2d}{2})} \\
        &&+\epsilon^{2-\kappa}\sup_{L\leq \frac{\epsilon d}{2}}L^{-\kappa}\inf_l\|U(t,x,\cdot) - l\|_{C_TB(x,L)}.
    \end{eqnarray*}
    The last term of the right-hand side can be absorbed in the left-hand side when $\epsilon$ is small enough. Thus we obtain the inequality $(\ref{three variable.step3})$.
    
    \textit{Step 4.} We prove that 
    \begin{equation}\label{three variable.step4}
        \begin{aligned}
            &\sup_{d\leq d_0} d^{\kappa}[U]_{C_T(\kappa,D_d,d)} \\
            &\lesssim \sup_{d\leq d_0}\sum_{\beta \in A}\left(M_{D_d,\frac{d}{2}}^{(1)}\epsilon^{-4+2\beta-\kappa} + M_{D_d,\frac{d}{2},\frac{\epsilon^2d}{2}}\epsilon^{\kappa-\beta}\right)+\epsilon^{2-2\kappa}\sup_{d\leq d_0}\|U\|_{C_T(D_d,d)}.
        \end{aligned}
    \end{equation}
    we first argue that we can change the order of the supremum and the infimum in $\sup_{R\leq \frac{\epsilon d}{2}} R^{-\kappa}\inf_l\|U(t,x,\cdot) - l\|_{C_TB(x,R)}$. To begin with, since $U(t,x,x) = 0$, we have 
    \begin{eqnarray*}
        \|U(t,x,\cdot) - (l-l(x))\|_{C_TB(x,R)} &\leq& \|U(t,x,\cdot) - l\|_{C_TB(x,R)} + |l(x)| \\
        &\leq& 2\|U(t,x,\cdot) - l\|_{C_TB(x,R)}.
    \end{eqnarray*}
    Hence 
    \[
        \sup_{R\leq \frac{\epsilon d}{2}} R^{-\kappa}\inf_{l(x) = 0}\|U(t,x,\cdot) - l\|_{C_TB(x,R)} \lesssim \sup_{R\leq \frac{\epsilon d}{2}} R^{-\kappa}\inf_l\|U(t,x,\cdot) - l\|_{C_TB(x,R)}.
    \]
    Furthermore, for $\frac{\epsilon d}{2}>R >0$, let $l_R = C_R(y-x)$ such that 
    \[
        \|U(t,x,\cdot) - l_R\|_{C_TB(x,R)} \leq 2\inf_l\|U(x,\cdot) - l\|_{C_TB(x,R)}.
    \]
    Then we have 
    \[
        |C_R - C_{\frac{R}{2}}|R^{-(\kappa - 1)} \lesssim \sup_{R\leq\frac{\epsilon d}{2}}R^{-\kappa}\inf_l\|U(t,x,\cdot) - l\|_{C_TB(x,R)}.
    \]
    This shows that there exists a limit $C_0:=\lim_{R\rightarrow 0}C_R$ and we have the bound 
    \[
        |C_R-C_0|R^{-(\kappa-1)} \lesssim_\kappa  \sup_{R\leq\frac{\epsilon d}{2}}R^{-\kappa}\inf_l\|U(t,x,\cdot) - l\|_{C_TB(x,R)}.
    \]
    Now consider $\bar{l} = C_0(y-x)$, we have 
    \[
        R^{-\kappa}\|U(t,x,\cdot) - \bar{l}\|_{C_TB(x,R)} \lesssim_\kappa  \sup_{R\leq\frac{\epsilon d}{2}}R^{-\kappa}\inf_l\|U(t,x,\cdot) - l\|_{C_TB(x,R)}.
    \]
    Thus for any $0\leq t\leq T$,
    \begin{eqnarray*}
        &&\inf_{\nu(t,x)}\sup_{x\neq y\in B(x,\frac{\epsilon d}{2})} d(x,y)^{-\kappa}|U(t,x,y) - \nu(t,x)(y-x)| \\
        &\leq& \inf_{l(x)=0}\sup_{R\leq \frac{\epsilon d}{2}} R^{-\kappa}\|U(t,x,\cdot) - l\|_{C_TB(x,R)}\\
        &\lesssim& \sup_{R\leq \frac{\epsilon d}{2}} R^{-\kappa}\inf_l\|U(t,x,\cdot) - l\|_{C_TB(x,R)}.
    \end{eqnarray*}
    Therefore, using the inequality (\ref{three variable.step3}),if we take the supremum over $t\in[0,T]$ and then take supremum over $x \in D_d$, multiply it by $d^\kappa$ and take the supremum over $d$, we will have 
    \begin{eqnarray*}
        &&\sup_{d\leq d_0}d^{\kappa}\sup_{x \in D_d} \sup_{0\leq t\leq T}\inf_{\nu(t,x)}\sup_{x\neq y\in B(x,\frac{\epsilon d}{2})} d(x,y)^{-\kappa}|U(t,x,y) - \nu(t,x)(y-x)| \\
        &\lesssim& \sup_{d\leq d_0}d^{\kappa}\sum_{\beta\in A}\left(M_{D_d,\frac{d}{2}}^{(1)}\epsilon^{-4+2\beta-\kappa} + M_{D_d,\frac{d}{2},\frac{\epsilon^2 d}{2}}^{(2)}\epsilon^{\kappa-\beta}\right) \\
        &&+ \epsilon^{2-2\kappa}\sup_{d\leq d_0}\|U\|_{C_T(D_d,d)}
        +\epsilon^\kappa\sup_{d\leq d_0}d^\kappa\sup_{x\in D_d}\sup_{y\in B(x,\frac{\epsilon d}{2})}\sup_{0\leq t\leq T}\\
        &&\left[\inf_{\nu(t,y)}\sup_{y\neq y_1\in B(y,\frac{\epsilon^2d}{2})}d(y,y_1)^{-\kappa}|U(y,y_1) - \nu(t,y)(y_1-y)|\right].
    \end{eqnarray*}
    The last term can be absorbed into the left-hand side for $\epsilon$ small enough since for $y \in B(x,\frac{\epsilon d}{2})$, we have $y \in D_{(1-\frac{\epsilon}{2})d}$ with $(1-\frac{\epsilon}{2}) \leq d_0$. And $\frac{\epsilon^2d}{2} \leq \frac{\epsilon}{2}(1-\frac{\epsilon}{2})d$ for small enough $\epsilon$. Hence we have 
    \begin{eqnarray*}
        \sup_{d\leq d_0}d^\kappa [U]_{C_T(\kappa,D_d,\frac{\epsilon d}{2})} &\lesssim_{T,\kappa}&\sup_{d\leq d_0}\sum_{\beta \in A}\left(M_{D_d,\frac{d}{2}}^{(1)}\epsilon^{-4+2\beta-\kappa} + M_{D_d,\frac{d}{2},\frac{\epsilon^2d}{2}}\epsilon^{\kappa-\beta}\right)\\
        &&+\epsilon^{2-2\kappa}\sup_{d\leq d_0}\|U\|_{C_T(D_d,d)}.
    \end{eqnarray*}
    Then it is directly to extend $[U]_{C_T(\kappa,D_d,\frac{\epsilon d}{2})}$ to $[U]_{C_T(\kappa,D_d)}$ by considering $\leq\frac{\epsilon d}{2}$ and $>\frac{\epsilon d}{2}$ parts.
\end{proof}

    \bibliographystyle{amsalpha}
	\bibliography{RSBM}
\end{document}